\documentclass[12pt]{amsart}

\usepackage{dcolumn}
\usepackage{color}
\usepackage{mathtools}
\usepackage{tikz}
\usepackage{graphicx}
\usepackage[shortlabels]{enumitem}
\usepackage{caption}
\usepackage{subcaption}
\usepackage[labelformat=simple]{subcaption}
\usepackage{tkz-euclide}
\usepackage{adjustbox}
\usepackage{changepage}
\usepackage{amsthm}
\usepackage{hyperref}
\usepackage{wasysym}
\usepackage{amssymb}
\usepackage{url}
\usepackage{filecontents}
\usepackage{dashbox}

\newtheorem{theorem}{Theorem}[section]
\newtheorem{myprop}[theorem]{Proposition}
\newtheorem{mylem}[theorem]{Lemma}
\newtheorem{mycor}[theorem]{Corollary}

\theoremstyle{definition}
\newtheorem{mydef}[theorem]{Definition}
\newtheorem{myrem}[theorem]{Remark}
\newtheorem{myex}[theorem]{Example}

\numberwithin{equation}{section}
\allowdisplaybreaks

\hypersetup{
  colorlinks   = true,
  urlcolor     = blue,
  linkcolor    = blue,
  citecolor   = red
}

\makeatletter
\def\namedlabel#1#2{\begingroup
    #2%
    \def\@currentlabel{#2}%
    \phantomsection\label{#1}\endgroup
}

\tikzset{anchorbase/.style={>=stealth,baseline={([yshift=-0.5ex]current bounding box.center)}}}

\newcommand\C{\mathbb{C}}
\newcommand\N{\mathbb{N}}
\newcommand\Z{\mathbb{Z}}

\newcommand\cB{\mathcal{B}}

\newcommand\cM{\mathcal{M}}

\newcommand\red{\mathrm{red}}

\newcommand\fh{\mathfrak{h}}
\newcommand\fsl{\mathfrak{sl}}

\newcommand\te{\tilde{e}}
\newcommand\tf{\tilde{f}}

\DeclareMathOperator{\wt}{wt}

\newcommand{\pref}[1]{(\ref{#1})}

\title{Crystals from 5-vertex ice models}

\author{J. Lorca Espiro}
\address[J. Lorca Espiro]{Departamento de F\'\i sica Matem\'atica do Instituto de F\'\i sica, Universidade de S\~{a}o Paulo \\
Department of Mathematics and Statistics, University of Ottawa \\
Departamento de Ciencias F\'\i sicas, Facultad de Ingenier\'\i a, Ciencias y Administraci\'on, Universidad de La Frontera}
\email{j.lorca.espiro@usp.br}

\author{Luke Volk}
\address[Luke Volk]{ Department of Mathematics and Statistics, University of Ottawa}
\email{lvolk005@uottawa.ca}


\begin{document}

\begin{abstract}
  Given a partition $\lambda$ corresponding to a dominant integral weight of $\fsl_n$, we define the structure of crystal on the set of 5-vertex ice models satisfying certain boundary conditions associated to $\lambda$.  We then show that the resulting crystal is isomorphic to that of the irreducible representation of highest weight $\lambda$.
\end{abstract}

\maketitle
\thispagestyle{empty}

\tableofcontents

%
\section{Introduction} \label{sec_Intro}
%

Six-vertex ice models were introduced by Linus Pauling as a method of studying crystals with hydrogen bonds, such as ice or potassium dihydrogen phosphate.  These models consist of grid graphs with edges labelled by spins (see Figure~\ref{icemodelfig}).  We refer the reader to \cite{Bax89} for an overview of these models in statistical mechanics.  The creation of this theoretical model allows for the application of mathematical tools coming from such diverse areas as combinatorics, number theory, representation theory, and dynamical systems.  For example, the partition function (a certain sum indexed by the possible states of the model) for six-vertex ice models satisfying certain boundary conditions has been computed in many cases, and related to such mathematical objects as Gelfand--Tsetlin patterns (see \cite{Tok88}) and Schur polynomials (see \cite{BBF11}).  See also \cite{BBCFSG12,BBB16,BBBF17} for further recent results.

In the current paper we consider $5$-\emph{vertex} ice models, derived from the six-vertex models by forbidding one vertex configuration.  To a partition $\lambda$, one can naturally associate a boundary condition for such ice models.  The set of models satisfying this boundary condition can be shown to be in bijection with semistandard Young tableaux of shape $\lambda$.  See \cite{Rokotoa}.

Young tableaux appear in many areas of combinatorics and representation theory.  Most important for the current paper is the fact that, given a partition $\lambda$, one can realize the crystal $\cB(\lambda)$ of the $\fsl_n$-representation of highest weight $\lambda$ in terms of Young tableaux of shape $\lambda$.  See, for example, the exposition in \cite[\S8.2]{hongintroduction}.  Combined with the bijection of \cite{Rokotoa}, this implies that one can define the structure of a crystal on the set $\cM(\lambda)$ of $5$-vertex ice models, with boundary condition given by $\lambda$, and obtain a crystal isomorphic to $\cB(\lambda)$.

The goal of the current paper is to precisely define a crystal structure on $\cM(\lambda)$ and show that the resulting crystal is isomorphic to $\cB(\lambda)$.  Instead of using the bijection with Young tableaux, our method of proof is direct.  We define the crystal structure explicitly and then verify directly that this crystal is indeed isomorphic to $\cB(\lambda)$, without ever referring to Young tableaux.  The key to our approach is a local characterization of crystals of simply-laced type due to Stembridge in \cite{Stembridge}.

The outline of the paper is as follows.  In Section~\ref{subsec:background}, we briefly recall the definition of $5$-vertex ice models and semiregular $\fsl_n$-crystals.  Then, in Section~\ref{sec_CrystalOnIce}, we endow the set $\cM(\lambda)$ with the structure of a crystal.  One of the key ingredients in Stembridge's characterization is the notion of a \emph{regular} crystal.  We show that $\cM(\lambda)$ is regular in Section~\ref{sec_Reg}.  Finally, we prove our main theorem, that $\cM(\lambda)$ is isomorphic to $\cB(\lambda)$, in Section~\ref{sec_CrystalM}.

The results of the current paper help further clarify the deep connection between ice models and representation theory.  Our hope is that it will lead to an increased level of understanding of the two fields, including further research inspired by developing other constructions in combinatorial representation theory (e.g.\ crystals in affine type $A$) in the language of ice models.

\subsection*{Acknowledgments}

We would like to thank Alistair Savage for his guidance and support, and for suggesting the topic of the current paper.  We also would like to thank Anthony Licata and Peter Tingley for their suggestions on how to proceed with certain proofs.

We also thank the University of Ottawa (J.L.E \& L.V.), the University of S\~{a}o Paulo (J.L.E.), and Universidad de la Frontera (J.L.E) for their logistical and technical support during the completion of this work.

%
\section{Background} \label{subsec:background}
%

\subsection{Ice models} \label{sec_IceModels}

In this section, we introduce our main object of study: $5$-vertex ice models.

\begin{mydef}[Ice model]
  For $n, s \in \N^+ = \{1,2,3,\dotsc\}$, an $n \times s$ \emph{ice model} consists of an $n \times s$ rectangular lattice and an assignment of exactly one sign (i.e.\ element of $\{+,-\}$) to each of the four edges adjacent to each vertex. The columns of an ice model are numbered from left to right $1, 2, \dotsc, s$ while the rows, from top to bottom, $n, n-1,\dotsc,1$. Figure \ref{icemodelfig} is an example of a $3 \times 5$ ice model.
\end{mydef}

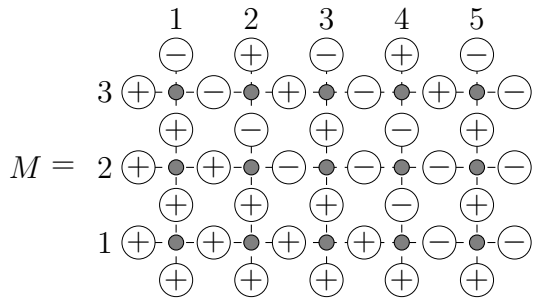
\begin{figure}[!h]
  \begin{center}
    \begin{tikzpicture}
      \node at (-4.5,0) {$M$};
      \node at (-4,0) {$=$};
      \draw [dashed, thin] (-3,-1) -- (2,-1);
      \draw [dashed, thin] (-3,0) -- (2,0);
      \draw [dashed, thin] (-3,1) -- (2,1);
      \draw [dashed, thin] (-2.5,-1.5) -- (-2.5,1.5);
      \draw [dashed, thin] (-1.5,-1.5) -- (-1.5,1.5);
      \draw [dashed, thin] (-0.5,-1.5) -- (-0.5,1.5);
      \draw [dashed, thin] (0.5,-1.5) -- (0.5,1.5);
      \draw [dashed, thin] (1.5,-1.5) -- (1.5,1.5);
      \draw [fill=white,ultra thin] (-3,-1) circle [radius=0.22];
      \node at (-3,-1) {$ + $};
      \node[left] at (-3.2,-1) {$1$};
      \draw [fill=white,ultra thin] (-3,0) circle [radius=0.22];
      \node at (-3,0) {$ + $};
      \node[left] at (-3.2,0) {$2$};
      \draw [fill=white,ultra thin] (-3,1) circle [radius=0.22];
      \node at (-3,1) {$ + $};
      \node[left] at (-3.2,1) {$3$};
      \draw [fill=white,ultra thin] (-2,-1) circle [radius=0.22];
      \node at (-2,-1) {$ + $};
      \draw [fill=white,ultra thin] (-2,0) circle [radius=0.22];
      \node at (-2,0) {$ + $};
      \draw [fill=white,ultra thin] (-2,1) circle [radius=0.22];
      \node at (-2,1) {$ - $};
      \draw [fill=white,ultra thin] (-1,-1) circle [radius=0.22];
      \node at (-1,-1) {$ + $};
      \draw [fill=white,ultra thin] (-1,0) circle [radius=0.22];
      \node at (-1,0) {$ - $};
      \draw [fill=white,ultra thin] (-1,1) circle [radius=0.22];
      \node at (-1,1) {$ + $};
      \draw [fill=white,ultra thin] (0,-1) circle [radius=0.22];
      \node at (0,-1) {$ + $};
      \draw [fill=white,ultra thin] (0,0) circle [radius=0.22];
      \node at (0,0) {$ - $};
      \draw [fill=white,ultra thin] (0,1) circle [radius=0.22];
      \node at (0,1) {$ - $};
      \draw [fill=white,ultra thin] (1,-1) circle [radius=0.22];
      \node at (1,-1) {$ - $};
      \draw [fill=white,ultra thin] (1,0) circle [radius=0.22];
      \node at (1,0) {$ - $};
      \draw [fill=white,ultra thin] (1,1) circle [radius=0.22];
      \node at (1,1) {$ + $};
      \draw [fill=white,ultra thin] (2,-1) circle [radius=0.22];
      \node at (2,-1) {$ - $};
      \draw [fill=white,ultra thin] (2,0) circle [radius=0.22];
      \node at (2,0) {$ - $};
      \draw [fill=white,ultra thin] (2,1) circle [radius=0.22];
      \node at (2,1) {$ - $};
      \draw [fill=white,ultra thin] (-2.5,-1.5) circle [radius=0.22];
      \node at (-2.5,-1.5) {$ + $};
      \draw [fill=white,ultra thin] (-2.5,-0.5) circle [radius=0.22];
      \node at (-2.5,-0.5) {$ + $};
      \draw [fill=white,ultra thin] (-2.5,0.5) circle [radius=0.22];
      \node at (-2.5,0.5) {$ + $};
      \draw [fill=white,ultra thin] (-2.5,1.5) circle [radius=0.22];
      \node at (-2.5,1.5) {$ - $};
      \node[above] at (-2.5,1.7) {$\small{1}$};
      \draw [fill=white,ultra thin] (-1.5,-1.5) circle [radius=0.22];
      \node at (-1.5,-1.5) {$ + $};
      \draw [fill=white,ultra thin] (-1.5,-0.5) circle [radius=0.22];
      \node at (-1.5,-0.5) {$ + $};
      \draw [fill=white,ultra thin] (-1.5,0.5) circle [radius=0.22];
      \node at (-1.5,0.5) {$ - $};
      \draw [fill=white,ultra thin] (-1.5,1.5) circle [radius=0.22];
      \node at (-1.5,1.5) {$ + $};
      \node[above] at (-1.5,1.7) {$\small{2}$};
      \draw [fill=white,ultra thin] (-0.5,-1.5) circle [radius=0.22];
      \node at (-0.5,-1.5) {$ + $};
      \draw [fill=white,ultra thin] (-0.5,-0.5) circle [radius=0.22];
      \node at (-0.5,-0.5) {$ + $};
      \draw [fill=white,ultra thin] (-0.5,0.5) circle [radius=0.22];
      \node at (-0.5,0.5) {$ + $};
      \draw [fill=white,ultra thin] (-0.5,1.5) circle [radius=0.22];
      \node at (-0.5,1.5) {$ - $};
      \node[above] at (-0.5,1.7) {$\small{3}$};
      \draw [fill=white,ultra thin] (0.5,-1.5) circle [radius=0.22];
      \node at (0.5,-1.5) {$ + $};
      \draw [fill=white,ultra thin] (0.5,-0.5) circle [radius=0.22];
      \node at (0.5,-0.5) {$ - $};
      \draw [fill=white,ultra thin] (0.5,0.5) circle [radius=0.22];
      \node at (0.5,0.5) {$ - $};
      \draw [fill=white,ultra thin] (0.5,1.5) circle [radius=0.22];
      \node at (0.5,1.5) {$ + $};
      \node[above] at (0.5,1.7) {$\small{4}$};
      \draw [fill=white,ultra thin] (1.5,-1.5) circle [radius=0.22];
      \node at (1.5,-1.5) {$ + $};
      \draw [fill=white,ultra thin] (1.5,-0.5) circle [radius=0.22];
      \node at (1.5,-0.5) {$ + $};
      \draw [fill=white,ultra thin] (1.5,0.5) circle [radius=0.22];
      \node at (1.5,0.5) {$ + $};
      \draw [fill=white,ultra thin] (1.5,1.5) circle [radius=0.22];
      \node at (1.5,1.5) {$ - $};
      \node[above] at (1.5,1.7) {$\small{5}$};
      \draw [fill=gray] (-2.5,-1) circle[radius=0.1];
      \draw [fill=gray] (-2.5,0) circle[radius=0.1];
      \draw [fill=gray] (-2.5,1) circle[radius=0.1];
      \draw [fill=gray] (-1.5,-1) circle[radius=0.1];
      \draw [fill=gray] (-1.5,0) circle[radius=0.1];
      \draw [fill=gray] (-1.5,1) circle[radius=0.1];
      \draw [fill=gray] (-0.5,-1) circle[radius=0.1];
      \draw [fill=gray] (-0.5,0) circle[radius=0.1];
      \draw [fill=gray] (-0.5,1) circle[radius=0.1];
      \draw [fill=gray] (0.5,-1) circle[radius=0.1];
      \draw [fill=gray] (0.5,0) circle[radius=0.1];
      \draw [fill=gray] (0.5,1) circle[radius=0.1];
      \draw [fill=gray] (1.5,-1) circle[radius=0.1];
      \draw [fill=gray] (1.5,0) circle[radius=0.1];
      \draw [fill=gray] (1.5,1) circle[radius=0.1];
    \end{tikzpicture}
  \end{center}
  \caption{A $3 \times 5$ ice model.\label{icemodelfig}}
\end{figure}

We will denote the vertices of an $n \times s$ ice model $M$ by $M_{i,j}$ for $1 \leq i \leq n$ and $1 \leq j \leq s$, where $M_{i,j},$ denotes the vertex in row $i$ and column $j$. For any particular vertex $M_{i,j}$ of an ice model $M$, we have signs on its top, left, right, and bottom edges, which we will denote using $M_{i,j}^x$ where $x \in \{\uparrow,\leftarrow,\rightarrow,\downarrow\}$ indicates the edge in the obvious way.  Although the $6$-vertex ice model is the most commonly studied variety, we will be focused on $5$-vertex ice models which are known to be in bijection with certain sets of semistandard Young tableaux (see \cite{Rokotoa}).

\begin{mydef}[Valid vertex configuration and $5$-vertex ice model]
  The five vertex configurations in Figure \ref{validvertices} are considered \textit{valid}  configurations. If all vertices of an ice model are valid, it is called a $5$-\textit{vertex ice model}.
  \begin{figure}[h]
      \begin{center}
        \begin{subfigure}[c]{0.15\textwidth}
   			  \begin{tikzpicture}
            \draw [dashed, thin] (-1,0) -- (1,0);
            \draw [dashed, thin] (0,-1) -- (0,1);
            \draw [fill=white,ultra thin] (0,-0.5) circle [radius=0.2];
            \node at (0,-0.5) {$ + $};
            \draw [fill=white,ultra thin] (-0.5,0) circle [radius=0.2];
            \node at (-0.5,0) {$ + $};
            \draw [fill=white,ultra thin] (0.5,0) circle [radius=0.2];
            \node at (0.5,0) {$ + $};
            \draw [fill=white,ultra thin] (0,0.5) circle [radius=0.2];
            \node at (0,0.5) {$ + $};
            \draw [fill=gray] (0,0) circle[radius=0.1];
  				\end{tikzpicture}
          \caption*{Type 1}
        \end{subfigure}
        \quad
        \begin{subfigure}[c]{0.15\textwidth}
   			  \begin{tikzpicture}
            \draw [dashed, thin] (-1,0) -- (1,0);
            \draw [dashed, thin] (0,-1) -- (0,1);
            \draw [fill=white,ultra thin] (0,-0.5) circle [radius=0.2];
            \node at (0,-0.5) {$ + $};
            \draw [fill=white,ultra thin] (-0.5,0) circle [radius=0.2];
            \node at (-0.5,0) {$ - $};
            \draw [fill=white,ultra thin] (0.5,0) circle [radius=0.2];
            \node at (0.5,0) {$ - $};
            \draw [fill=white,ultra thin] (0,0.5) circle [radius=0.2];
            \node at (0,0.5) {$ + $};
            \draw [fill=gray] (0,0) circle[radius=0.1];
  				\end{tikzpicture}
          \caption*{Type 2 (box)}
        \end{subfigure}
        \quad
        \begin{subfigure}[c]{0.15\textwidth}
          \begin{tikzpicture}
            \draw [dashed, thin] (-1,0) -- (1,0);
            \draw [dashed, thin] (0,-1) -- (0,1);
            \draw [fill=white,ultra thin] (0,-0.5) circle [radius=0.2];
            \node at (0,-0.5) {$ - $};
            \draw [fill=white,ultra thin] (-0.5,0) circle [radius=0.2];
            \node at (-0.5,0) {$ - $};
            \draw [fill=white,ultra thin] (0.5,0) circle [radius=0.2];
            \node at (0.5,0) {$ + $};
            \draw [fill=white,ultra thin] (0,0.5) circle [radius=0.2];
            \node at (0,0.5) {$ + $};
            \draw [fill=gray] (0,0) circle[radius=0.1];
  				\end{tikzpicture}
          \caption*{Type 3}
        \end{subfigure}
        \quad
        \begin{subfigure}[c]{0.15\textwidth}
          \begin{tikzpicture}
            \draw [dashed, thin] (-1,0) -- (1,0);
            \draw [dashed, thin] (0,-1) -- (0,1);
            \draw [fill=white,ultra thin] (0,-0.5) circle [radius=0.2];
            \node at (0,-0.5) {$ + $};
            \draw [fill=white,ultra thin] (-0.5,0) circle [radius=0.2];
            \node at (-0.5,0) {$ + $};
            \draw [fill=white,ultra thin] (0.5,0) circle [radius=0.2];
            \node at (0.5,0) {$ - $};
            \draw [fill=white,ultra thin] (0,0.5) circle [radius=0.2];
            \node at (0,0.5) {$ - $};
            \draw [fill=gray] (0,0) circle[radius=0.1];
  				\end{tikzpicture}
          \caption*{Type 4}
          \end{subfigure}
        \quad
        \begin{subfigure}[c]{0.15\textwidth}
          \begin{tikzpicture}
            \draw [dashed, thin] (-1,0) -- (1,0);
            \draw [dashed, thin] (0,-1) -- (0,1);
            \draw [fill=white,ultra thin] (0,-0.5) circle [radius=0.2];
            \node at (0,-0.5) {$ - $};
            \draw [fill=white,ultra thin] (-0.5,0) circle [radius=0.2];
            \node at (-0.5,0) {$ - $};
            \draw [fill=white,ultra thin] (0.5,0) circle [radius=0.2];
            \node at (0.5,0) {$ - $};
            \draw [fill=white,ultra thin] (0,0.5) circle [radius=0.2];
            \node at (0,0.5) {$ - $};
            \draw [fill=gray] (0,0) circle[radius=0.1];
  				\end{tikzpicture}
          \caption*{Type 5}
        \end{subfigure}
      \end{center}
    \caption{The five valid vertex configurations of a $5$-vertex ice model\label{validvertices}}
  \end{figure}
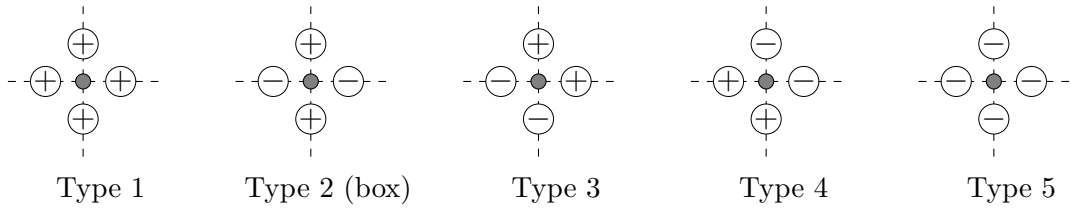
\end{mydef}

For brevity, all further uses of the term ``ice model'' will refer to $5$-vertex ice models only.

Recall that a partition is a tuple $\lambda = (\lambda_1, \lambda_2, \dotsc, \lambda_n)$ such that $\lambda_1 \geq \lambda_2 \geq \dotsc \geq \lambda_n = 0$.  Note that we allow some of the parts to be equal to zero and we force $\lambda_n=0$ since we will later want to associate partitions with weights of $\fsl_n$.  For each partition $\lambda$, there is a family of ice models satisfying a boundary condition determined by $\lambda$.

\begin{mydef}[$\cM(\lambda)$] \label{def_boundary}
  Suppose $\lambda = (\lambda_1, \lambda_2, \dotsc, \lambda_n)$ is a partition. We define $\cM(\lambda)$ to be the collection of $n\times(\lambda_1 + n)$ ice models $M$ such that:
  \begin{enumerate}
    \item \label{def:boundary-top-minuses} $M_{n,q}^\uparrow = -$ if and only if $q = \lambda_1 + j -\lambda_{j}$ for some $j \in \left\{ n, n-1 , \cdots , 1 \right\}$
    \item	$M_{p,1}^\leftarrow = +$ for all $1\leq p\leq n$;
    \item	$M_{1,q}^\downarrow = +$ for all $1 \leq q\leq \lambda_1 + n$;
    \item	$M_{p,\lambda_1 + n}^\rightarrow = -$ for all $1\leq p\leq n$.
  \end{enumerate}
\end{mydef}

The ice model $M$ shown in Figure~\ref{icemodelfig} is an element of $\cM(\lambda)$ for $\lambda = (2 , 1 , 0)$.

\subsection{Crystals} \label{sec_Crystals}

In this section we briefly recall the notion of crystals. Although the theory of crystals is developed in much greater generality, we restrict our attention here to semiregular crystals of finite type $A$.  We refer the reader to \cite{hongintroduction} for further details.

Consider the simple Lie algebra $\fsl_n$ over the field $\C$ of complex numbers and define $I = \{1,2,\dotsc,n-1\}$.  Let $E_{i,j}$ denote the matrix with a $1$ in the $(i,j)$-position and $0$ elsewhere, and define
\[
  h_i = E_{i,i} - E_{i+1,i+1},\quad i \in I.
\]
Then the $h_i$ span the standard Cartan subalgebra $\fh$ of $\fsl_n$.  For $i \in I$, define
\[
  \epsilon'_i \colon M_{n \times n}(\C) \to \C,\quad
  \epsilon'_i(B) = B_{i,i},
\]
where $M_{n \times n}(\C)$ denotes the space of $n \times n$ complex matrices.  We let $\epsilon_i$ denote the restriction of $\epsilon_i'$ to $\fh$.  Then the weight lattice of $\fsl_n$ is given by
\[
  P = \left(\bigoplus_{i=1}^{n-1} \Z\epsilon_i\right)/\Z(\epsilon_1 + \epsilon_2 + \dotsb + \epsilon_{n}) \subseteq \fh^*.
\]
We also have the simple roots $ \quad \alpha_i := \epsilon_i - \epsilon_{i+1},\quad i \in I.$

If $\langle \cdot, \cdot \rangle \colon \fh \times \fh^* \to \C$ denotes the canonical pairing, we have
\begin{equation} \label{eq:Cartan-matrix-entries}
  \langle h_j, \alpha_i \rangle = 2 \delta_{i,j} - \delta_{i+1,j} - \delta_{i,j+1},
\end{equation}
where $\delta_{k,\ell}$ denotes the Kronecker delta.

\medskip

The dominant integral weight lattice
\[
  P_+ = \bigoplus_{i=1}^{n-1} \N \alpha_i
\]
is naturally identified with the set of partitions $\lambda = (\lambda_1,\dotsc,\lambda_n)$, $\lambda_1 \ge \dotsb \ge \lambda_{n-1} \ge \lambda_n = 0$, of length at most $n-1$.  In particular, the partition $\lambda$ is identified with $\sum_{i=1}^{n-1} \lambda_i \epsilon_i$.

\begin{mydef}[Crystal] \label{crystal}
  A \emph{semiregular $\fsl_n$-crystal} is a set $\cB$ along with maps
  \[
    \wt \colon \cB \to P,\quad
    \te_i, \tf_i \colon \cB \to \cB \cup \{0\},
  \]
  such that, if we define
  \begin{gather}
    \varepsilon_i \colon \cB \to \N,\quad
    \varepsilon_i(b) = \max \{n \in \N \mid \te_i^n(b) \in \cB\}, \label{eq:varepsilon_i-def} \\
    \varphi_i \colon \cB \to \N,\quad
    \varphi_i(b) = \max \{m \in \N \mid \tf_i^m(b) \in \cB\}, \label{eq:varphi_i-def}
  \end{gather}
  then, for all $b,b' \in \cB$ and $i \in I$, the following conditions are satisfied:
  \begin{description}[style=multiline, labelwidth=1cm]
    \item[\namedlabel{C1}{C1}] $\varphi_i( b ) = \varepsilon_i( b ) + \langle h_i,\wt( b )\rangle$;
    \item[\namedlabel{C2}{C2}] $\wt(\te_i( b )) = \wt( b ) + \alpha_i$, when $\te_i( b )\in\cB$;
    \item[\namedlabel{C3}{C3}] $\wt(\tf_i( b )) = \wt( b ) - \alpha_i$, when $\tf_i( b )\in\cB$;
    \item[\namedlabel{C4}{C4}] $\varepsilon_i(\te_i( b )) = \varepsilon_i( b ) - 1$ and $\varphi_i(\te_i( b )) = \varphi_i( b ) + 1$, when $\te_i( b )\in\cB$;
    \item[\namedlabel{C5}{C5}] $\varepsilon_i(\tf_i( b )) = \varepsilon_i( b ) + 1$ and $\varphi_i(\tf_i( b )) = \varphi_i( b ) - 1$, when $\tf_i( b )\in\cB$;
    \item[\namedlabel{C6}{C6}] $\tf_i( b ) = b'$ if and only if $b = \te_i( b' )$;
  \end{description}
\end{mydef}

Throughout this paper, we will use the term \emph{crystal} to mean semiregular $\fsl_n$-crystal.  For every dominant integral weight (i.e.\ partition) $\lambda$, let $\cB(\lambda)$ be the crystal associated to the irreducible $\fsl_n$-module of highest weight $\lambda$.  The goal of the current paper is to define a crystal structure on the set $\cM(\lambda)$ of ice models, and identify the resulting crystal with $\cB(\lambda)$.  For the remainder of the paper, we fix a partition $\lambda = (\lambda_1, \lambda_2, \dotsc, \lambda_n)$, $\lambda_1 \ge \dotsb \ge \lambda_n = 0$.

%
\section{Crystal structure on ice models} \label{sec_CrystalOnIce}
%

In this section, we define a crystal structure on the set $\cM(\lambda)$ of ice models satisfying the boundary condition determined by the partition $\lambda$.

\subsection{The boxing map} \label{sec_BoxingMap}

We first introduce an important ingredient that we have called the boxing map.  The boxing map will allow us to better visualize the crystal structure.  Recall the valid vertex configurations of Figure~\ref{validvertices}.

\begin{mydef}[Boxing map]
  For $M \in \cM(\lambda)$, define
  \[
    \beta(M) = \{ (p,q) \mid 1 \le x \le n,\ 1 \le y \le \lambda_1 + n,\ M_{p,q} \text{ is of type 2} \}.
  \]
  We call $\beta$ the \emph{boxing map}. If $(p,q) \in \beta(M)$, then $M_{p,q}$ will be called a \emph{box} of $M$.
\end{mydef}

Pictorially, for an ice model $M \in \cM(\lambda)$, we indicate a box in $M$ by boxing the vertex.  If all boxes of $M$ have been indicated in this manner, we say that $M$ is \emph{boxed}.  Boxing the ice model $M$ of Figure~\ref{icemodelfig} yields the picture in Figure~\ref{boxedicemodel}

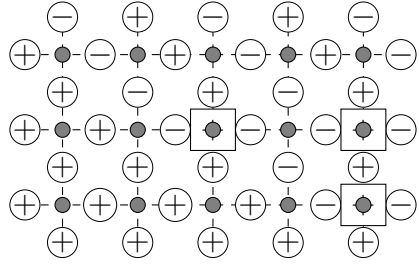
\begin{figure}[!h]
  \begin{center}
    \begin{tikzpicture}
      \draw [dashed, thin] (-3,-1) -- (2,-1);
      \draw [dashed, thin] (-3,0) -- (2,0);
      \draw [dashed, thin] (-3,1) -- (2,1);
      \draw [dashed, thin] (-2.5,-1.5) -- (-2.5,1.5);
      \draw [dashed, thin] (-1.5,-1.5) -- (-1.5,1.5);
      \draw [dashed, thin] (-0.5,-1.5) -- (-0.5,1.5);
      \draw [dashed, thin] (0.5,-1.5) -- (0.5,1.5);
      \draw [dashed, thin] (1.5,-1.5) -- (1.5,1.5);
      \draw [fill=white,ultra thin] (-3,-1) circle [radius=0.2];
      \node at (-3,-1) {$ + $};
      \draw [fill=white,ultra thin] (-3,0) circle [radius=0.2];
      \node at (-3,0) {$ + $};
      \draw [fill=white,ultra thin] (-3,1) circle [radius=0.2];
      \node at (-3,1) {$ + $};
      \draw [fill=white,ultra thin] (-2,-1) circle [radius=0.22];
      \node at (-2,-1) {$ + $};
      \draw [fill=white,ultra thin] (-2,0) circle [radius=0.2];
      \node at (-2,0) {$ + $};
      \draw [fill=white,ultra thin] (-2,1) circle [radius=0.2];
      \node at (-2,1) {$ - $};
      \draw [fill=white,ultra thin] (-1,-1) circle [radius=0.22];
      \node at (-1,-1) {$ + $};
      \draw [fill=white,ultra thin] (-1,0) circle [radius=0.2];
      \node at (-1,0) {$ - $};
      \draw [fill=white,ultra thin] (-1,1) circle [radius=0.2];
      \node at (-1,1) {$ + $};
      \draw [fill=white,ultra thin] (0,-1) circle [radius=0.2];
      \node at (0,-1) {$ + $};
      \draw [fill=white,ultra thin] (0,0) circle [radius=0.2];
      \node at (0,0) {$ - $};
      \draw [fill=white,ultra thin] (0,1) circle [radius=0.2];
      \node at (0,1) {$ - $};
      \draw [fill=white,ultra thin] (1,-1) circle [radius=0.2];
      \node at (1,-1) {$ - $};
      \draw [fill=white,ultra thin] (1,0) circle [radius=0.2];
      \node at (1,0) {$ - $};
      \draw [fill=white,ultra thin] (1,1) circle [radius=0.2];
      \node at (1,1) {$ + $};
      \draw [fill=white,ultra thin] (2,-1) circle [radius=0.2];
      \node at (2,-1) {$ - $};
      \draw [fill=white,ultra thin] (2,0) circle [radius=0.2];
      \node at (2,0) {$ - $};
      \draw [fill=white,ultra thin] (2,1) circle [radius=0.2];
      \node at (2,1) {$ - $};
      \draw [fill=white,ultra thin] (-2.5,-1.5) circle [radius=0.2];
      \node at (-2.5,-1.5) {$ + $};
      \draw [fill=white,ultra thin] (-2.5,-0.5) circle [radius=0.2];
      \node at (-2.5,-0.5) {$ + $};
      \draw [fill=white,ultra thin] (-2.5,0.5) circle [radius=0.2];
      \node at (-2.5,0.5) {$ + $};
      \draw [fill=white,ultra thin] (-2.5,1.5) circle [radius=0.2];
      \node at (-2.5,1.5) {$ - $};
      \draw [fill=white,ultra thin] (-1.5,-1.5) circle [radius=0.2];
      \node at (-1.5,-1.5) {$ + $};
      \draw [fill=white,ultra thin] (-1.5,-0.5) circle [radius=0.2];
      \node at (-1.5,-0.5) {$ + $};
      \draw [fill=white,ultra thin] (-1.5,0.5) circle [radius=0.2];
      \node at (-1.5,0.5) {$ - $};
      \draw [fill=white,ultra thin] (-1.5,1.5) circle [radius=0.2];
      \node at (-1.5,1.5) {$ + $};
      \draw [fill=white,ultra thin] (-0.5,-1.5) circle [radius=0.2];
      \node at (-0.5,-1.5) {$ + $};
      \draw [fill=white,ultra thin] (-0.5,-0.5) circle [radius=0.2];
      \node at (-0.5,-0.5) {$ + $};
      \draw [fill=white,ultra thin] (-0.5,0.5) circle [radius=0.2];
      \node at (-0.5,0.5) {$ + $};
      \node at (-0.5,0) [rectangle,draw,fill=white] {$+$};
      \draw [fill=white,ultra thin] (-0.5,1.5) circle [radius=0.2];
      \node at (-0.5,1.5) {$ - $};
      \draw [fill=white,ultra thin] (0.5,-1.5) circle [radius=0.2];
      \node at (0.5,-1.5) {$ + $};
      \draw [fill=white,ultra thin] (0.5,-0.5) circle [radius=0.2];
      \node at (0.5,-0.5) {$ - $};
      \draw [fill=white,ultra thin] (0.5,0.5) circle [radius=0.2];
      \node at (0.5,0.5) {$ - $};
      \draw [fill=white,ultra thin] (0.5,1.5) circle [radius=0.2];
      \node at (0.5,1.5) {$ + $};
      \draw [fill=white,ultra thin] (1.5,-1.5) circle [radius=0.2];
      \node at (1.5,-1.5) {$ + $};
      \draw [fill=white,ultra thin] (1.5,-0.5) circle [radius=0.2];
      \node at (1.5,-0.5) {$ + $};
      \node at (1.5,-1) [rectangle,draw,fill=white] {$+$};
      \draw [fill=white,ultra thin] (1.5,0.5) circle [radius=0.2];
      \node at (1.5,0.5) {$ + $};
      \node at (1.5,0) [rectangle,draw,fill=white] {$+$};
      \draw [fill=white,ultra thin] (1.5,1.5) circle [radius=0.2];
      \node at (1.5,1.5) {$ - $};
      \draw [fill=gray] (-2.5,-1) circle[radius=0.1];
      \draw [fill=gray] (-2.5,0) circle[radius=0.1];
      \draw [fill=gray] (-2.5,1) circle[radius=0.1];
      \draw [fill=gray] (-1.5,-1) circle[radius=0.1];
      \draw [fill=gray] (-1.5,0) circle[radius=0.1];
      \draw [fill=gray] (-1.5,1) circle[radius=0.1];
      \draw [fill=gray] (-0.5,-1) circle[radius=0.1];
      \draw [fill=gray] (-0.5,0) circle[radius=0.1];
      \draw [fill=gray] (-0.5,1) circle[radius=0.1];
      \draw [fill=gray] (0.5,-1) circle[radius=0.1];
      \draw [fill=gray] (0.5,0) circle[radius=0.1];
      \draw [fill=gray] (0.5,1) circle[radius=0.1];
      \draw [fill=gray] (1.5,-1) circle[radius=0.1];
      \draw [fill=gray] (1.5,0) circle[radius=0.1];
      \draw [fill=gray] (1.5,1) circle[radius=0.1];
    \end{tikzpicture}
  \end{center}
  \caption{\label{boxedicemodel} A boxed $3 \times 5$ ice model.}
\end{figure}

The following lemma gives an important condition on the positions of the boxes in an ice model, in particular that a box never occurs in the first column of the ice model.

\begin{mylem} \label{lemma_noboxin0}
  Suppose $M\in\cM(\lambda)$ and $(i,j)\in\beta(M)$. Then $j > 1$.
\end{mylem}

\begin{proof}
  Let $M\in\cM(\lambda)$ and $(i,j)\in\beta(M)$. Since $(i,j)$ is a box, $M_{i,j}^\leftarrow = -$, and hence it would be a contradiction to the left boundary conditions of $\cM(\lambda)$ if $j=1$.  Thus $j>1$.
\end{proof}

\subsection{Crystal structure: definition} \label{subsec:cystal-structure-def}

We now define the maps giving the structure of a crystal on the set $\cM(\lambda)$.  We begin by defining an analog for ice models of the so-called $i$-signature commonly appearing in realizations of crystals.

\begin{mydef}[$i$-signature] \label{def_isig}
  Given an $M\in\cM(\lambda)$ and an $i \in I$, we consider the set $\beta \left( M \right)_i = \{(p,q)\in\beta(M) \mid p\in\{i,i+1\}\}$ with the lexicographic total order:
  \[
    (p,q)\preceq(p',q') \Longleftrightarrow (q < q') \quad \text{or} \quad ((q = q') \quad \text{and} \quad (p \leq p')).
  \]
  Let $m = |\beta \left( M \right)_i|$. For $1\leq j\leq m$ let $(p_j,q_j)$ be the $j^\text{th}$ element of $\beta \left( M \right)_i$ under the total order $\preceq$, and define:
  \[
    w_j =
    \begin{cases}
    \leftmoon & \text{if } p_j=i, \\
    \rightmoon & \text{if } p_j=i+1.
    \end{cases}
  \]
  The \emph{$i$-signature} of $M$, denoted $\sigma_i(M)$, is the word $w=w_1 w_2w_3 \dotsm w_m$. The \emph{reduced} $i$-signature of $M$, denoted by $\sigma^\red_i(M)$, is obtained from the $i$-signature $\sigma_i \left( M \right)$ by removing all sub-words $\leftmoon \rightmoon$ until the word is reduced to the form $\rightmoon^m \; \leftmoon^n$, for some $m,n\in\mathbb{N}$.
\end{mydef}

\begin{mylem} \label{lem_localchange}
  Let $M\in\cM(\lambda)$ and $i \in I$.
  \begin{enumerate}[(i)]
    \item	If $u\in\beta(M)$ corresponds to the first $\leftmoon$ in $\sigma^\red_i(M)$, then locally $M$ appears as in Figure~\ref{phipic}.  Furthermore, if $M$ is changed locally as in Figure~\ref{prop2fig2b}, the result is another ice model.	\label{lem2i}
    \item	If $v\in\beta(M)$ corresponds to the last $\rightmoon$ in $\sigma^\red_i(M)$, then locally $M$ appears as in Figure~\ref{eipic}.	 Furthermore, if $M$ is changed locally as in Figure~\ref{prop2fig3a}, the result is another ice model. \label{lem2ii}
  \end{enumerate}

  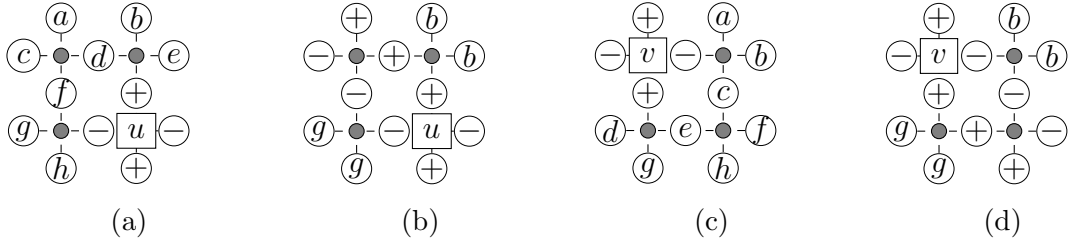
\begin{figure}[!h]
    \begin{center}
      \begin{subfigure}[c]{0.2\textwidth}
 			  \begin{tikzpicture}
          \draw [dashed, thin] (-3,-1) -- (-1,-1);
          \draw [dashed, thin] (-3,0) -- (-1,0);
          \draw [dashed, thin] (-2.5,-1.5) -- (-2.5,0.5);
          \draw [dashed, thin] (-1.5,-1.5) -- (-1.5,0.5);
          \draw [fill=white,ultra thin] (-3,-1) circle [radius=0.2];
          \node at (-3,-1) {$ g $};
          \draw [fill=white,ultra thin] (-3,0) circle [radius=0.22];
          \node at (-3,0) {$ c $};
          \draw [fill=white,ultra thin] (-2,-1) circle [radius=0.2];
          \node at (-2,-1) {$ - $};
          \draw [fill=white,ultra thin] (-2,0) circle [radius=0.2];
          \node at (-2,0) {$ d $};
          \draw [fill=white,ultra thin] (-1,-1) circle [radius=0.2];
          \node at (-1,-1) {$ - $};
          \draw [fill=white,ultra thin] (-1,0) circle [radius=0.2];
          \node at (-1,0) {$ e $};
          \draw [fill=white,ultra thin] (-1.5,-1.5) circle [radius=0.2];
          \node at (-1.5,-1.5) {$ + $};
          \draw [fill=white,ultra thin] (-1.5,-0.5) circle [radius=0.2];
          \node at (-1.5,-0.5) {$ + $};
          \draw [fill=white,ultra thin] (-1.5,0.5) circle [radius=0.2];
          \node at (-1.5,0.5) {$ b $};
          \draw [fill=white,ultra thin] (-2.5,-1.5) circle [radius=0.2];
          \node at (-2.5,-1.5) {$ h $};
          \draw [fill=white,ultra thin] (-2.5,-0.5) circle [radius=0.2];
          \node at (-2.5,-0.5) {$ f $};
          \draw [fill=white,ultra thin] (-2.5,0.5) circle [radius=0.2];
          \node at (-2.5,0.5) {$ a $};
          \draw [fill=gray] (-2.5,-1) circle[radius=0.1];
          \draw [fill=gray] (-2.5,0) circle[radius=0.1];
          \draw [fill=gray] (-1.5,0) circle[radius=0.1];		
          \node at (-1.5,-1) [rectangle,draw,fill=white] {$u$};
        \end{tikzpicture}
				\caption{\label{phipic0}}
      \end{subfigure}
      \quad
      \begin{subfigure}[c]{0.2\textwidth}
        \begin{tikzpicture}
          \draw [dashed, thin] (-3,-1) -- (-1,-1);
          \draw [dashed, thin] (-3,0) -- (-1,0);
          \draw [dashed, thin] (-2.5,-1.5) -- (-2.5,0.5);
          \draw [dashed, thin] (-1.5,-1.5) -- (-1.5,0.5);
          \draw [fill=white,ultra thin] (-3,-1) circle [radius=0.2];
          \node at (-3,-1) {$ g $};
          \draw [fill=white,ultra thin] (-3,0) circle [radius=0.2];
          \node at (-3,0) {$ - $};
          \draw [fill=white,ultra thin] (-2,-1) circle [radius=0.2];
          \node at (-2,-1) {$ - $};
          \draw [fill=white,ultra thin] (-2,0) circle [radius=0.2];
          \node at (-2,0) {$ + $};
          \draw [fill=white,ultra thin] (-1,-1) circle [radius=0.2];
          \node at (-1,-1) {$ - $};
          \draw [fill=white,ultra thin] (-1,0) circle [radius=0.2];
          \node at (-1,0) {$ b $};
          \draw [fill=white,ultra thin] (-1.5,-1.5) circle [radius=0.2];
          \node at (-1.5,-1.5) {$ + $};
          \draw [fill=white,ultra thin] (-1.5,-0.5) circle [radius=0.2];
          \node at (-1.5,-0.5) {$ + $};
          \draw [fill=white,ultra thin] (-1.5,0.5) circle [radius=0.2];
          \node at (-1.5,0.5) {$ b $};
          \draw [fill=white,ultra thin] (-2.5,-1.5) circle [radius=0.2];
          \node at (-2.5,-1.5) {$ g $};
          \draw [fill=white,ultra thin] (-2.5,-0.5) circle [radius=0.2];
          \node at (-2.5,-0.5) {$ - $};
          \draw [fill=white,ultra thin] (-2.5,0.5) circle [radius=0.2];
          \node at (-2.5,0.5) {$ + $};
          \draw [fill=gray] (-2.5,-1) circle[radius=0.1];
          \draw [fill=gray] (-2.5,0) circle[radius=0.1];
          \draw [fill=gray] (-1.5,0) circle[radius=0.1];		
          \node at (-1.5,-1) [rectangle,draw,fill=white]{$ u $};
        \end{tikzpicture}
        \caption{\label{phipic}}
      \end{subfigure}
      \quad
      \begin{subfigure}[c]{0.2\textwidth}
        \begin{tikzpicture}
          \draw [dashed, thin] (-3,-1) -- (-1,-1);
          \draw [dashed, thin] (-3,0) -- (-1,0);
          \draw [dashed, thin] (-2.5,-1.5) -- (-2.5,0.5);
          \draw [dashed, thin] (-1.5,-1.5) -- (-1.5,0.5);
          \draw [fill=white,ultra thin] (-3,-1) circle [radius=0.2];
          \node at (-3,-1) {$ d $};
          \draw [fill=white,ultra thin] (-3,0) circle [radius=0.2];
          \node at (-3,0) {$ - $};
          \draw [fill=white,ultra thin] (-2,-1) circle [radius=0.2];
          \node at (-2,-1) {$ e $};
          \draw [fill=white,ultra thin] (-2,0) circle [radius=0.2];
          \node at (-2,0) {$ - $};
          \draw [fill=white,ultra thin] (-1,-1) circle [radius=0.2];
          \node at (-1,-1) {$ f $};
          \draw [fill=white,ultra thin] (-1,0) circle [radius=0.2];
          \node at (-1,0) {$ b $};
          \draw [fill=white,ultra thin] (-1.5,-1.5) circle [radius=0.2];
          \node at (-1.5,-1.5) {$ h $};
          \draw [fill=white,ultra thin] (-1.5,-0.5) circle [radius=0.2];
          \node at (-1.5,-0.5) {$ c $};
          \draw [fill=white,ultra thin] (-1.5,0.5) circle [radius=0.2];
          \node at (-1.5,0.5) {$ a $};
          \draw [fill=white,ultra thin] (-2.5,-1.5) circle [radius=0.2];
          \node at (-2.5,-1.5) {$ g $};
          \draw [fill=white,ultra thin] (-2.5,-0.5) circle [radius=0.2];
          \node at (-2.5,-0.5) {$ + $};
          \draw [fill=white,ultra thin] (-2.5,0.5) circle [radius=0.2];
          \node at (-2.5,0.5) {$ + $};
          \node at (-2.5,0) [rectangle,draw,fill=white]{$ v $};
          \draw [fill=gray] (-2.5,-1) circle[radius=0.1];
          \draw [fill=gray] (-1.5,-1) circle[radius=0.1];
          \draw [fill=gray] (-1.5,0) circle[radius=0.1];	
        \end{tikzpicture}
        \caption{\label{eipi0}}
      \end{subfigure}
      \quad
      \begin{subfigure}[c]{0.2\textwidth}
  			\begin{tikzpicture}
          \draw [dashed, thin] (-3,-1) -- (-1,-1);
          \draw [dashed, thin] (-3,0) -- (-1,0);
          \draw [dashed, thin] (-2.5,-1.5) -- (-2.5,0.5);
          \draw [dashed, thin] (-1.5,-1.5) -- (-1.5,0.5);
          \draw [fill=white,ultra thin] (-3,-1) circle [radius=0.2];
          \node at (-3,-1) {$ g $};
          \draw [fill=white,ultra thin] (-3,0) circle [radius=0.2];
          \node at (-3,0) {$ - $};
          \draw [fill=white,ultra thin] (-2,-1) circle [radius=0.2];
          \node at (-2,-1) {$ + $};
          \draw [fill=white,ultra thin] (-2,0) circle [radius=0.2];
          \node at (-2,0) {$ - $};
          \draw [fill=white,ultra thin] (-1,-1) circle [radius=0.2];
          \node at (-1,-1) {$ - $};
          \draw [fill=white,ultra thin] (-1,0) circle [radius=0.2];
          \node at (-1,0) {$ b $};
          \draw [fill=white,ultra thin] (-1.5,-1.5) circle [radius=0.2];
          \node at (-1.5,-1.5) {$ + $};
          \draw [fill=white,ultra thin] (-1.5,-0.5) circle [radius=0.2];
          \node at (-1.5,-0.5) {$ - $};
          \draw [fill=white,ultra thin] (-1.5,0.5) circle [radius=0.2];
          \node at (-1.5,0.5) {$ b $};
          \draw [fill=white,ultra thin] (-2.5,-1.5) circle [radius=0.2];
          \node at (-2.5,-1.5) {$ g $};
          \draw [fill=white,ultra thin] (-2.5,-0.5) circle [radius=0.2];
          \node at (-2.5,-0.5) {$ + $};
          \draw [fill=white,ultra thin] (-2.5,0.5) circle [radius=0.2];
          \node at (-2.5,0.5) {$ + $};
          \node at (-2.5,0) [rectangle,draw,fill=white]{$ v $};
          \draw [fill=gray] (-2.5,-1) circle[radius=0.1];
          \draw [fill=gray] (-1.5,-1) circle[radius=0.1];
          \draw [fill=gray] (-1.5,0) circle[radius=0.1];	
        \end{tikzpicture}
        \caption{\label{eipic}}
      \end{subfigure}
      \caption{The local configuration around the box $u$ is shown in (b), and the local configuration around the box $v$ is shown in (d).}
    \end{center}
  \end{figure}

  \begin{figure}[!h]
    \begin{center}
      \begin{subfigure}[c]{0.45\textwidth}
 			  \begin{tikzpicture}[anchorbase]
          \draw [dashed, thin] (-3,-1) -- (-1,-1);
          \draw [dashed, thin] (-3,0) -- (-1,0);
          \draw [dashed, thin] (-2.5,-1.5) -- (-2.5,0.5);
          \draw [dashed, thin] (-1.5,-1.5) -- (-1.5,0.5);
          \draw [fill=white,ultra thin] (-3,-1) circle [radius=0.2];
          \node at (-3,-1) {$ g $};
          \draw [fill=white,ultra thin] (-3,0) circle [radius=0.2];
          \node at (-3,0) {$ - $};
          \draw [fill=white,ultra thin] (-2,-1) circle [radius=0.2];
          \node at (-2,-1) {$ - $};
          \draw [fill=white,ultra thin] (-2,0) circle [radius=0.2];
          \node at (-2,0) {$ + $};
          \draw [fill=white,ultra thin] (-1,-1) circle [radius=0.2];
          \node at (-1,-1) {$ - $};
          \draw [fill=white,ultra thin] (-1,0) circle [radius=0.2];
          \node at (-1,0) {$ b $};
          \draw [fill=white,ultra thin] (-1.5,-1.5) circle [radius=0.2];
          \node at (-1.5,-1.5) {$ + $};
          \draw [fill=white,ultra thin] (-1.5,-0.5) circle [radius=0.2];
          \node at (-1.5,-0.5) {$ + $};
          \draw [fill=white,ultra thin] (-1.5,0.5) circle [radius=0.2];
          \node at (-1.5,0.5) {$ b $};
          \draw [fill=white,ultra thin] (-2.5,-1.5) circle [radius=0.2];
          \node at (-2.5,-1.5) {$ g $};
          \draw [fill=white,ultra thin] (-2.5,-0.5) circle [radius=0.2];
          \node at (-2.5,-0.5) {$ - $};
          \draw [fill=white,ultra thin] (-2.5,0.5) circle [radius=0.2];
          \node at (-2.5,0.5) {$ + $};
          \node at (-1.5,-1) [rectangle,draw,fill=white]{$ u $};
          \draw [fill=gray] (-2.5,-1) circle[radius=0.1];
          \draw [fill=gray] (-2.5,0) circle[radius=0.1];
          \draw [fill=gray] (-1.5,0) circle[radius=0.1];	
        \end{tikzpicture}
        $\stackrel{\tf_i}{\rightsquigarrow}$
        \begin{tikzpicture}[anchorbase]
          \draw [dashed, thin] (-3,-1) -- (-1,-1);
          \draw [dashed, thin] (-3,0) -- (-1,0);
          \draw [dashed, thin] (-2.5,-1.5) -- (-2.5,0.5);
          \draw [dashed, thin] (-1.5,-1.5) -- (-1.5,0.5);
          \draw [fill=white,ultra thin] (-3,-1) circle [radius=0.2];
          \node at (-3,-1) {$ g $};
          \draw [fill=white,ultra thin] (-3,0) circle [radius=0.2];
          \node at (-3,0) {$ - $};
          \draw [fill=white,ultra thin] (-2,-1) circle [radius=0.2];
          \node at (-2,-1) {$ + $};
          \draw [fill=white,ultra thin] (-2,0) circle [radius=0.2];
          \node at (-2,0) {$ - $};
          \draw [fill=white,ultra thin] (-1,-1) circle [radius=0.2];
          \node at (-1,-1) {$ - $};
          \draw [fill=white,ultra thin] (-1,0) circle [radius=0.2];
          \node at (-1,0) {$ b $};
          \draw [fill=white,ultra thin] (-1.5,-1.5) circle [radius=0.2];
          \node at (-1.5,-1.5) {$ + $};
          \draw [fill=white,ultra thin] (-1.5,-0.5) circle [radius=0.2];
          \node at (-1.5,-0.5) {$ - $};
          \draw [fill=white,ultra thin] (-1.5,0.5) circle [radius=0.2];
          \node at (-1.5,0.5) {$ b $};
          \draw [fill=white,ultra thin] (-2.5,-1.5) circle [radius=0.2];
          \node at (-2.5,-1.5) {$ g $};
          \draw [fill=white,ultra thin] (-2.5,-0.5) circle [radius=0.2];
          \node at (-2.5,-0.5) {$ + $};
          \draw [fill=white,ultra thin] (-2.5,0.5) circle [radius=0.2];
          \node at (-2.5,0.5) {$ + $};
          \node at (-2.5,0) [rectangle,draw,fill=white]{$+$};
          \draw [fill=gray] (-2.5,-1) circle[radius=0.1];
          \draw [fill=gray] (-2.5,0) circle[radius=0.1];
          \draw [fill=gray] (-1.5,-1) circle[radius=0.1];
          \draw [fill=gray] (-1.5,0) circle[radius=0.1];		
        \end{tikzpicture}
        \caption{\label{prop2fig2b}}
      \end{subfigure}
      \quad
      \begin{subfigure}[c]{0.45\textwidth}
        \begin{tikzpicture}[anchorbase]
          \draw [dashed, thin] (-3,-1) -- (-1,-1);
          \draw [dashed, thin] (-3,0) -- (-1,0);
          \draw [dashed, thin] (-2.5,-1.5) -- (-2.5,0.5);
          \draw [dashed, thin] (-1.5,-1.5) -- (-1.5,0.5);
          \draw [fill=white,ultra thin] (-3,-1) circle [radius=0.2];
          \node at (-3,-1) {$ g $};
          \draw [fill=white,ultra thin] (-3,0) circle [radius=0.2];
          \node at (-3,0) {$ - $};
          \draw [fill=white,ultra thin] (-2,-1) circle [radius=0.2];
          \node at (-2,-1) {$ + $};
          \draw [fill=white,ultra thin] (-2,0) circle [radius=0.2];
          \node at (-2,0) {$ - $};
          \draw [fill=white,ultra thin] (-1,-1) circle [radius=0.2];
          \node at (-1,-1) {$ - $};
          \draw [fill=white,ultra thin] (-1,0) circle [radius=0.2];
          \node at (-1,0) {$ b $};
          \draw [fill=white,ultra thin] (-1.5,-1.5) circle [radius=0.2];
          \node at (-1.5,-1.5) {$ + $};
          \draw [fill=white,ultra thin] (-1.5,-0.5) circle [radius=0.2];
          \node at (-1.5,-0.5) {$ - $};
          \draw [fill=white,ultra thin] (-1.5,0.5) circle [radius=0.2];
          \node at (-1.5,0.5) {$ b $};
          \draw [fill=white,ultra thin] (-2.5,-1.5) circle [radius=0.2];
          \node at (-2.5,-1.5) {$ g $};
          \draw [fill=white,ultra thin] (-2.5,-0.5) circle [radius=0.2];
          \node at (-2.5,-0.5) {$ + $};
          \draw [fill=white,ultra thin] (-2.5,0.5) circle [radius=0.2];
          \node at (-2.5,0.5) {$ + $};
          \node at (-2.5,0) [rectangle,draw,fill=white]{$ v $};
          \draw [fill=gray] (-2.5,-1) circle[radius=0.1];
          \draw [fill=gray] (-1.5,-1) circle[radius=0.1];
          \draw [fill=gray] (-1.5,0) circle[radius=0.1];		
        \end{tikzpicture}
        $\stackrel{\te_i}{\rightsquigarrow}$
        \begin{tikzpicture}[anchorbase]
          \draw [dashed, thin] (-3,-1) -- (-1,-1);
          \draw [dashed, thin] (-3,0) -- (-1,0);
          \draw [dashed, thin] (-2.5,-1.5) -- (-2.5,0.5);
          \draw [dashed, thin] (-1.5,-1.5) -- (-1.5,0.5);
          \draw [fill=white,ultra thin] (-3,-1) circle [radius=0.2];
          \node at (-3,-1) {$ g $};
          \draw [fill=white,ultra thin] (-3,0) circle [radius=0.2];
          \node at (-3,0) {$ - $};
          \draw [fill=white,ultra thin] (-2,-1) circle [radius=0.2];
          \node at (-2,-1) {$ - $};
          \draw [fill=white,ultra thin] (-2,0) circle [radius=0.2];
          \node at (-2,0) {$ + $};
          \draw [fill=white,ultra thin] (-1,-1) circle [radius=0.2];
          \node at (-1,-1) {$ - $};
          \draw [fill=white,ultra thin] (-1,0) circle [radius=0.2];
          \node at (-1,0) {$ b $};
          \draw [fill=white,ultra thin] (-1.5,-1.5) circle [radius=0.2];
          \node at (-1.5,-1.5) {$ + $};
          \draw [fill=white,ultra thin] (-1.5,-0.5) circle [radius=0.2];
          \node at (-1.5,-0.5) {$ + $};
          \draw [fill=white,ultra thin] (-1.5,0.5) circle [radius=0.2];
          \node at (-1.5,0.5) {$ b $};
          \draw [fill=white,ultra thin] (-2.5,-1.5) circle [radius=0.2];
          \node at (-2.5,-1.5) {$ g $};
          \draw [fill=white,ultra thin] (-2.5,-0.5) circle [radius=0.2];
          \node at (-2.5,-0.5) {$ - $};
          \draw [fill=white,ultra thin] (-2.5,0.5) circle [radius=0.2];
          \node at (-2.5,0.5) {$ + $};
          \node at (-1.5,-1) [rectangle,draw,fill=white]{$+$};
          \draw [fill=gray] (-2.5,-1) circle[radius=0.1];
          \draw [fill=gray] (-2.5,0) circle[radius=0.1];
          \draw [fill=gray] (-1.5,-1) circle[radius=0.1];
          \draw [fill=gray] (-1.5,0) circle[radius=0.1];	
        \end{tikzpicture}
  			\caption{\label{prop2fig3a}}
      \end{subfigure}
      \caption{Figures (a) and (b) depict the local change in $M$ in Lemma~\ref{lem_localchange}\ref{lem2i} and \ref{lem2ii}, respectively.\label{fig_localfi}}
    \end{center}
  \end{figure}
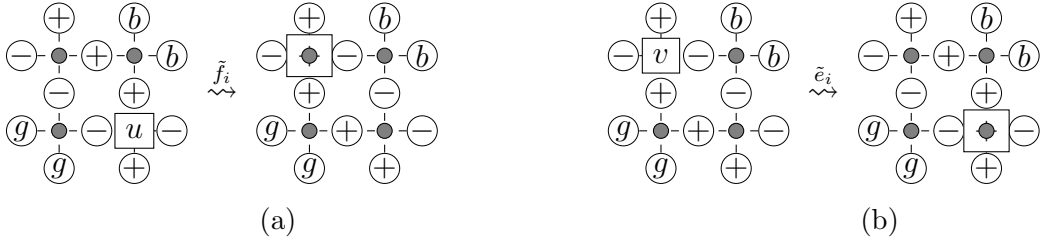
\end{mylem}

\begin{proof}
  Let $M\in\cM(\lambda)$ and $\sigma^\red_i \left( M \right) = a_1 a_2\dotsc a_p b_1 b_2\dotsc b_q$ where $a_j = \rightmoon$ and $b_j = \leftmoon$ for $p,q\in\mathbb{N}$.

  \smallskip

  \noindent\emph{Proof of \ref{lem2i}:}

  Suppose $q\geq 1$ and $b_1$ corresponds to a box $u \in \beta \left( M \right)$. Locally, we have a configuration as depicted in Figure~\ref{phipic0}.

  If $f=+$, then the only valid configuration for the bottom-left vertex is type 2 (i.e.\ a box). Since the $\leftmoon$ corresponding to $f$ is not in $\sigma^\red_i(M)$, it must be cancelled through the top-left vertex being a box as well. But if this is the case, $d=-$ which implies that the top-right vertex is also a box. If this were the case, the $\leftmoon$ corresponding to $u$ would be cancelled in $\sigma^\red_i(M)$, a contradiction. Hence $f = -$.

  If $a=-$, the only valid vertex configuration for the top-left vertex is type 5. As above, the top-right vertex would then be a box and the same contradiction follows.  Hence $a=+$, which implies that the top-left vertex is of type 3, implying that the top-right vertex is of type 1 or 4.  In addition the bottom-left vertex is of type 4 or 5.  In particular, $b=e$ and $g=h$, being of either sign.  Thus, $M$ is locally as in Figure \ref{phipic}.

  The statement concerning a local change of $M$ is clear since one can directly verify that the vertices in Figure~\ref{prop2fig2b} have valid configurations and the signs around the perimeter are unchanged.

  \smallskip

  \noindent\emph{Proof of \ref{lem2ii}:}

  Suppose $p\geq 1$ and $a_p$ corresponds to a box $v \in \beta(M)$. Locally, we have the configuration as depicted in Figure \ref{eipi0}.

  If $e=-$, then the bottom-left vertex is of type 2 (i.e.\ a box).  This implies that $\leftmoon\rightmoon$ occurs in $\sigma_i(M)$, the $\rightmoon$ associated with $v$, a contradiction.  Hence $e=+$, immediately implying the bottom-right vertex is of type 1 or 4.  Hence $h=+$.

  If $c=+$, then the top-right vertex is a box and $f=+$.  Since the bottom two vertices are not boxes, we encounter a contradiction since, if the top-right vertex were a box, it would mean that the associated $\rightmoon$ would necessarily cancel with a $\leftmoon$ resulting from either the bottom-left or right vertices, otherwise $v$ does not correspond to $a_p$ in $\sigma^\red_i(M)$. Hence $c=f=-$.  As in the proof of part~\ref{lem2i}, we have $a=b$ and $d=g$, being of either sign.  Thus, $M$ is locally as in Figure~\ref{eipic}.

  Again, the statement concerning a local change of $M$ is clear since one can directly verify that the vertices in Figure~\ref{prop2fig3a} have valid configurations and the signs around the perimeter are unchanged.
\end{proof}

We are now able to define the crystal operators $\te_i$ and $\tf_j$ on $M(\lambda)$.

\begin{mydef}[Crystal operators $\te_i$, $\tf_i$] \label{def_fiei}
  Suppose $M\in\cM(\lambda)$ and $i \in I$.  We define the \emph{crystal operators}
  \[
    \tf_i, \te_i \colon \cM(\lambda) \to \cM(\lambda) \cup \{0\}
  \]
  as follows.
  \begin{enumerate}
    \item If $\leftmoon$ does not occur in $\sigma^\red_i(M)$, then $\tf_i(M) = 0$.  Otherwise, we let $u \in \beta(M)$ be the vertex corresponding to the first $\leftmoon$ in $\sigma_i(M)$ and define
      \[
        \tf_i(M) = N\in\cM(\lambda),
      \]
      where $N$ is the ice model resulting from the local change of Figure~\ref{prop2fig2b}.

    \item If $\rightmoon$ does not occur in $\sigma^\red_i(M)$, then $\te_{i}(M) = 0$.  Otherwise we let $v \in \beta(M)$ be the vertex corresponding to the last $\rightmoon$ in $\sigma_i^\red(M)$ and define
        \[
          \te_i(M) = N\in\cM(\lambda),
        \]
        where $N$ is the ice model resulting from the local change of Figure~\ref{prop2fig3a}.
  \end{enumerate}
\end{mydef}

\begin{mylem} \label{lem_boxmoved}
  \begin{enumerate}
    \item If $u=(i,q) \in \beta(M)$ is the vertex corresponding to the first $\leftmoon$ in $\sigma^\red_i(M)$, then
      \[
        \beta(\tf_i(M)) = \big( \beta(M)\cup\{(i+1,q-1)\} \big) \setminus \{(i,q)\}.
      \]
    \item If $v=(i+1,q) \in \beta(M)$ is the vertex corresponding to the last $\rightmoon$ in $\sigma^\red_i(M)$, then
      \[
        \beta(\te_i(M)) = \big( \beta(M)\cup\{(i,q+1)\} \big) \setminus \{(i+1,q)\}.
      \]
  \end{enumerate}
\end{mylem}

\begin{proof}
  This follows immediately from Definition~\ref{def_fiei} and Figure~\ref{fig_localfi}.
\end{proof}

As a result of Lemma~\ref{lem_boxmoved}, any application of the crystal operators to an ice model $M$---assuming it is nonzero---preserve the cardinality of $\beta(M)$.  Because of this, we say that the box in $M$ that has been removed after application of a crystal operator has \emph{moved} to the newly added box.  If the $\leftmoon$ or $\rightmoon$ corresponding to a box $b\in \beta(M)$ survives in $\sigma^\red_i(M)$ (i.e.\ is not cancelled) for some $i \in I$, then we say that it is \emph{$i$-movable}.  We also say that the box itself is $i$-movable.

The application of a crystal operators can be thought of as a way of locally ``moving the boxes'' of a boxed ice model.  In particular, the $\te_i$ move boxes down and to the right, while the $\tf_i$ move boxes up and to the left.

\begin{myex}
  Consider the boxed ice model $M$ on the left-hand side of Figure~\ref{ex3}.  Note that  $\sigma_1(M) = \sigma^\red_1(M) = \rightmoon \; \leftmoon$, hence $\te_1(M) \neq 0$.  Applying $\te_1$ moves the box in the second row down and to the right.  On the other hand, $\sigma_2(M) = \leftmoon \rightmoon$, and so $\sigma^\red_2(M)$ is then simply the empty word.  Thus $\tf_2(M) = \te_2(M) = 0$.
  \begin{figure}[!h]
    \begin{center}
      \begin{tikzpicture}[anchorbase]
        \draw [dashed, thin] (-3,-1) -- (2,-1);
        \draw [dashed, thin] (-3,0) -- (2,0);
        \draw [dashed, thin] (-3,1) -- (2,1);
        \draw [dashed, thin] (-2.5,-1.5) -- (-2.5,1.5);
        \draw [dashed, thin] (-1.5,-1.5) -- (-1.5,1.5);
        \draw [dashed, thin] (-0.5,-1.5) -- (-0.5,1.5);
        \draw [dashed, thin] (0.5,-1.5) -- (0.5,1.5);
        \draw [dashed, thin] (1.5,-1.5) -- (1.5,1.5);
        \draw [fill=white,ultra thin] (-3,-1) circle [radius=0.2];
        \node at (-3,-1) {$ + $};
        \draw [fill=white,ultra thin] (-3,0) circle [radius=0.2];
        \node at (-3,0) {$ + $};
        \draw [fill=white,ultra thin] (-3,1) circle [radius=0.2];
        \node at (-3,1) {$ + $};
        \draw [fill=white,ultra thin] (-2,-1) circle [radius=0.2];
        \node at (-2,-1) {$ + $};
        \draw [fill=white,ultra thin] (-2,0) circle [radius=0.2];
        \node at (-2,0) {$ + $};
        \draw [fill=white,ultra thin] (-2,1) circle [radius=0.2];
        \node at (-2,1) {$ - $};
        \draw [fill=white,ultra thin] (-1,-1) circle [radius=0.2];
        \node at (-1,-1) {$ + $};
        \draw [fill=white,ultra thin] (-1,0) circle [radius=0.2];
        \node at (-1,0) {$ - $};
        \draw [fill=white,ultra thin] (-1,1) circle [radius=0.2];
        \node at (-1,1) {$ + $};
        \draw [fill=white,ultra thin] (0,-1) circle [radius=0.2];
        \node at (0,-1) {$ + $};
        \draw [fill=white,ultra thin] (0,0) circle [radius=0.2];
        \node at (0,0) {$ - $};
        \draw [fill=white,ultra thin] (0,1) circle [radius=0.2];
        \node at (0,1) {$ - $};
        \draw [fill=white,ultra thin] (1,-1) circle [radius=0.2];
        \node at (1,-1) {$ - $};
        \draw [fill=white,ultra thin] (1,0) circle [radius=0.2];
        \node at (1,0) {$ + $};
        \draw [fill=white,ultra thin] (1,1) circle [radius=0.2];
        \node at (1,1) {$ - $};
        \draw [fill=white,ultra thin] (2,-1) circle [radius=0.2];
        \node at (2,-1) {$ - $};
        \draw [fill=white,ultra thin] (2,0) circle [radius=0.2];
        \node at (2,0) {$ - $};
        \draw [fill=white,ultra thin] (2,1) circle [radius=0.2];
        \node at (2,1) {$ - $};
        \draw [fill=white,ultra thin] (-2.5,-1.5) circle [radius=0.2];
        \node at (-2.5,-1.5) {$ + $};
        \draw [fill=white,ultra thin] (-2.5,-0.5) circle [radius=0.2];
        \node at (-2.5,-0.5) {$ + $};
        \draw [fill=white,ultra thin] (-2.5,0.5) circle [radius=0.2];
        \node at (-2.5,0.5) {$ + $};
        \draw [fill=white,ultra thin] (-2.5,1.5) circle [radius=0.2];
        \node at (-2.5,1.5) {$ - $};
        \draw [fill=white,ultra thin] (-1.5,-1.5) circle [radius=0.2];
        \node at (-1.5,-1.5) {$ + $};
        \draw [fill=white,ultra thin] (-1.5,-0.5) circle [radius=0.2];
        \node at (-1.5,-0.5) {$ + $};
        \draw [fill=white,ultra thin] (-1.5,0.5) circle [radius=0.2];
        \node at (-1.5,0.5) {$ - $};
        \draw [fill=white,ultra thin] (-1.5,1.5) circle [radius=0.2];
        \node at (-1.5,1.5) {$ + $};
        \draw [fill=white,ultra thin] (-0.5,-1.5) circle [radius=0.2];
        \node at (-0.5,-1.5) {$ + $};
        \draw [fill=white,ultra thin] (-0.5,-0.5) circle [radius=0.2];
        \node at (-0.5,-0.5) {$ + $};
        \draw [fill=white,ultra thin] (-0.5,0.5) circle [radius=0.2];
        \node at (-0.5,0.5) {$ + $};
        \draw [fill=white,ultra thin] (-0.5,1.5) circle [radius=0.2];
        \node at (-0.5,1.5) {$ - $};
        \draw [fill=white,ultra thin] (0.5,-1.5) circle [radius=0.2];
        \node at (0.5,-1.5) {$ + $};
        \draw [fill=white,ultra thin] (0.5,-0.5) circle [radius=0.2];
        \node at (0.5,-0.5) {$ - $};
        \draw [fill=white,ultra thin] (0.5,0.5) circle [radius=0.2];
        \node at (0.5,0.5) {$ + $};
        \draw [fill=white,ultra thin] (0.5,1.5) circle [radius=0.2];
        \node at (0.5,1.5) {$ + $};
        \draw [fill=white,ultra thin] (1.5,-1.5) circle [radius=0.2];
        \node at (1.5,-1.5) {$ + $};
        \draw [fill=white,ultra thin] (1.5,-0.5) circle [radius=0.2];
        \node at (1.5,-0.5) {$ + $};
        \draw [fill=white,ultra thin] (1.5,0.5) circle [radius=0.2];
        \node at (1.5,0.5) {$ - $};
        \draw [fill=white,ultra thin] (1.5,1.5) circle [radius=0.2];
        \node at (1.5,1.5) {$ - $};
        \node at (-0.5,0) [rectangle,draw,fill=white] {$+$};
        \node at (0.5,1) [rectangle,draw,fill=white]  {$+$};
        \node at (1.5,-1) [rectangle,draw,fill=white] {$+$};
        \draw [fill=gray] (-2.5,-1) circle[radius=0.1];
        \draw [fill=gray] (-2.5,0) circle[radius=0.1];
        \draw [fill=gray] (-2.5,1) circle[radius=0.1];
        \draw [fill=gray] (-1.5,-1) circle[radius=0.1];
        \draw [fill=gray] (-1.5,0) circle[radius=0.1];
        \draw [fill=gray] (-1.5,1) circle[radius=0.1];
        \draw [fill=gray] (-0.5,-1) circle[radius=0.1];
        \draw [fill=gray] (-0.5,0) circle[radius=0.1];
        \draw [fill=gray] (-0.5,1) circle[radius=0.1];
        \draw [fill=gray] (0.5,-1) circle[radius=0.1];
        \draw [fill=gray] (0.5,0) circle[radius=0.1];
        \draw [fill=gray] (0.5,1) circle[radius=0.1];
        \draw [fill=gray] (1.5,-1) circle[radius=0.1];
        \draw [fill=gray] (1.5,0) circle[radius=0.1];
        \draw [fill=gray] (1.5,1) circle[radius=0.1];
      \end{tikzpicture}
      \quad $\xmapsto{\te_1}$ \quad
 			\begin{tikzpicture}[anchorbase]
        \draw [dashed, thin] (-3,-1) -- (2,-1);
        \draw [dashed, thin] (-3,0) -- (2,0);
        \draw [dashed, thin] (-3,1) -- (2,1);
        \draw [dashed, thin] (-2.5,-1.5) -- (-2.5,1.5);
        \draw [dashed, thin] (-1.5,-1.5) -- (-1.5,1.5);
        \draw [dashed, thin] (-0.5,-1.5) -- (-0.5,1.5);
        \draw [dashed, thin] (0.5,-1.5) -- (0.5,1.5);
        \draw [dashed, thin] (1.5,-1.5) -- (1.5,1.5);
        \draw [fill=white,ultra thin] (-3,-1) circle [radius=0.2];
        \node at (-3,-1) {$ + $};
        \draw [fill=white,ultra thin] (-3,0) circle [radius=0.2];
        \node at (-3,0) {$ + $};
        \draw [fill=white,ultra thin] (-3,1) circle [radius=0.2];
        \node at (-3,1) {$ + $};
        \draw [fill=white,ultra thin] (-2,-1) circle [radius=0.2];
        \node at (-2,-1) {$ + $};
        \draw [fill=white,ultra thin] (-2,0) circle [radius=0.2];
        \node at (-2,0) {$ + $};
        \draw [fill=white,ultra thin] (-2,1) circle [radius=0.2];
        \node at (-2,1) {$ - $};
        \draw [fill=white,ultra thin] (-1,-1) circle [radius=0.2];
        \node at (-1,-1) {$ + $};
        \draw [fill=white,ultra thin] (-1,0) circle [radius=0.2];
        \node at (-1,0) {$ - $};
        \draw [fill=white,ultra thin] (-1,1) circle [radius=0.2];
        \node at (-1,1) {$ + $};
        \draw [fill=white,ultra thin] (0,-1) circle [radius=0.2];
        \node at (0,-1) {$ - $};
        \draw [fill=white,ultra thin] (0,0) circle [radius=0.2];
        \node at (0,0) {$ + $};
        \draw [fill=white,ultra thin] (0,1) circle [radius=0.2];
        \node at (0,1) {$ - $};
        \draw [fill=white,ultra thin] (1,-1) circle [radius=0.2];
        \node at (1,-1) {$ - $};
        \draw [fill=white,ultra thin] (1,0) circle [radius=0.2];
        \node at (1,0) {$ + $};
        \draw [fill=white,ultra thin] (1,1) circle [radius=0.2];
        \node at (1,1) {$ - $};
        \draw [fill=white,ultra thin] (2,-1) circle [radius=0.2];
        \node at (2,-1) {$ - $};
        \draw [fill=white,ultra thin] (2,0) circle [radius=0.2];
        \node at (2,0) {$ - $};
        \draw [fill=white,ultra thin] (2,1) circle [radius=0.2];
        \node at (2,1) {$ - $};
        \draw [fill=white,ultra thin] (-2.5,-1.5) circle [radius=0.2];
        \node at (-2.5,-1.5) {$ + $};
        \draw [fill=white,ultra thin] (-2.5,-0.5) circle [radius=0.2];
        \node at (-2.5,-0.5) {$ + $};
        \draw [fill=white,ultra thin] (-2.5,0.5) circle [radius=0.2];
        \node at (-2.5,0.5) {$ + $};
        \draw [fill=white,ultra thin] (-2.5,1.5) circle [radius=0.2];
        \node at (-2.5,1.5) {$ - $};
        \draw [fill=white,ultra thin] (-1.5,-1.5) circle [radius=0.2];
        \node at (-1.5,-1.5) {$ + $};
        \draw [fill=white,ultra thin] (-1.5,-0.5) circle [radius=0.2];
        \node at (-1.5,-0.5) {$ + $};
        \draw [fill=white,ultra thin] (-1.5,0.5) circle [radius=0.2];
        \node at (-1.5,0.5) {$ - $};
        \draw [fill=white,ultra thin] (-1.5,1.5) circle [radius=0.2];
        \node at (-1.5,1.5) {$ + $};
        \draw [fill=white,ultra thin] (-0.5,-1.5) circle [radius=0.2];
        \node at (-0.5,-1.5) {$ + $};
        \draw [fill=white,ultra thin] (-0.5,-0.5) circle [radius=0.2];
        \node at (-0.5,-0.5) {$ - $};
        \draw [fill=white,ultra thin] (-0.5,0.5) circle [radius=0.2];
        \node at (-0.5,0.5) {$ + $};
        \draw [fill=white,ultra thin] (-0.5,1.5) circle [radius=0.2];
        \node at (-0.5,1.5) {$ - $};
        \draw [fill=white,ultra thin] (0.5,-1.5) circle [radius=0.2];
        \node at (0.5,-1.5) {$ + $};
        \draw [fill=white,ultra thin] (0.5,-0.5) circle [radius=0.2];
        \node at (0.5,-0.5) {$ + $};
        \draw [fill=white,ultra thin] (0.5,0.5) circle [radius=0.2];
        \node at (0.5,0.5) {$ + $};
        \draw [fill=white,ultra thin] (0.5,1.5) circle [radius=0.2];
        \node at (0.5,1.5) {$ + $};
        \draw [fill=white,ultra thin] (1.5,-1.5) circle [radius=0.2];
        \node at (1.5,-1.5) {$ + $};
        \draw [fill=white,ultra thin] (1.5,-0.5) circle [radius=0.2];
        \node at (1.5,-0.5) {$ + $};
        \draw [fill=white,ultra thin] (1.5,0.5) circle [radius=0.2];
        \node at (1.5,0.5) {$ - $};
        \draw [fill=white,ultra thin] (1.5,1.5) circle [radius=0.2];
        \node at (1.5,1.5) {$ - $};
        \node at (0.5,-1) [rectangle,draw,fill=white] {$+$};
        \node at (1.5,-1) [rectangle,draw,fill=white] {$+$};
        \node at (0.5,1) [rectangle,draw,fill=white] {$+$};
        \draw [fill=gray] (-2.5,-1) circle[radius=0.1];
        \draw [fill=gray] (-2.5,0) circle[radius=0.1];
        \draw [fill=gray] (-2.5,1) circle[radius=0.1];
        \draw [fill=gray] (-1.5,-1) circle[radius=0.1];
        \draw [fill=gray] (-1.5,0) circle[radius=0.1];
        \draw [fill=gray] (-1.5,1) circle[radius=0.1];
        \draw [fill=gray] (-0.5,-1) circle[radius=0.1];
        \draw [fill=gray] (-0.5,0) circle[radius=0.1];
        \draw [fill=gray] (-0.5,1) circle[radius=0.1];
        \draw [fill=gray] (0.5,-1) circle[radius=0.1];
        \draw [fill=gray] (0.5,0) circle[radius=0.1];
        \draw [fill=gray] (0.5,1) circle[radius=0.1];
        \draw [fill=gray] (1.5,-1) circle[radius=0.1];
        \draw [fill=gray] (1.5,0) circle[radius=0.1];
        \draw [fill=gray] (1.5,1) circle[radius=0.1];
      \end{tikzpicture}
      \caption{The application of the crystal operator $\te_1$ moves the position of a box.\label{ex3}}
    \end{center}
  \end{figure}
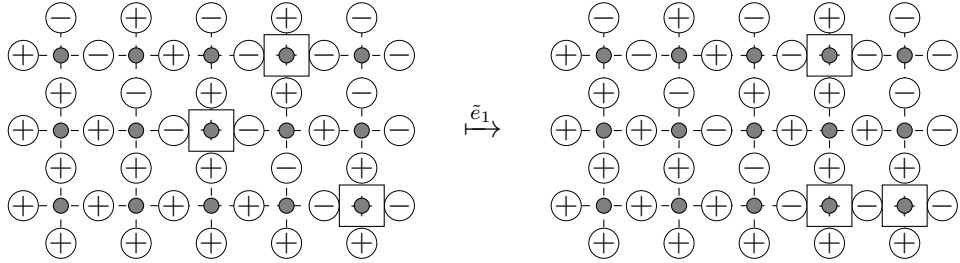
\end{myex}

\begin{myrem} \label{rem:finite-crystal-operator}
  It follows immediately from the above discussion that one can only apply the operators $\te_i$ (respectively, $\tf_i$) a finite number of times before obtaining zero.
\end{myrem}

\begin{mydef}[Weight function $\wt$] \label{def_wt}
  We define the \emph{weight function}
  \[
    \wt \colon \cM(\lambda) \to P,\qquad
    \wt(M) = \sum_{i = 1}^n a^{M}_i\epsilon_i,\quad
    a^{M}_i = |\{(p,q)\in\beta(M) \mid p = i\}|.
  \]
\end{mydef}

\subsection{Crystal structure: verification}

We now verify that the maps defined in Section~\ref{subsec:cystal-structure-def} do indeed endow $\cM(\lambda)$ with the structure of a crystal.

\begin{mylem} \label{lem_eifiswitch}
  Let $M \in \cM(\lambda)$ and $i \in I$.  If $\tf_i$ moves $u \in \beta(M)$ to $v \in \beta(\tf_iM)$, then $\te_i$ moves $v$ to $u \in \beta(\te_i\tf_i M)$.  Similarly, if $\te_i$ moves $u \in \beta(M)$ to $v \in \beta(\te_i M)$, then $\tf_i$ moves $v$ to $u \in \beta(\tf_i \te_i M)$.
\end{mylem}

\begin{proof}
  Assume that $\tf_i$ moves $u\in \beta(M)$ to $v \in \beta(\tf_i M)$.  Suppose, towards a contradiction, that $v$ is not $i$-movable.  This implies that there is a box $w \in \beta(\tf_iM)$ in row $i$ (which was $i$-movable in $M$) such that $w \preceq v$.  But we know $v \preceq u$ and thus $w\preceq u$.  Hence $\tf_i$ would have moved $w$ instead of $u$, a contradiction.  Hence $v$ is $i$-movable.

  Now suppose, towards a contradiction, that there exists an $i$-movable box $w \in \beta(\tf_iM)$ in row $i+1$ such that $v \preceq w$. But this means the $\rightmoon$ corresponding to $w$ cancels with the $\leftmoon$ corresponding to $u$ in $\sigma^\red_i(M)$, a contradiction.  Thus an application of $\te_i$ to $\tf_i(M)$ moves $v$ to $u$.

  The proof of the final statement of the lemma is analogous.
\end{proof}

\begin{mycor} \label{cor_defepsilonphi}
  For all $M \in \cM(\lambda)$ and $i \in I$, $\varepsilon_i(M)$ is the number of $\rightmoon$ that occur in $\sigma^\red_i(M)$, and $\varphi_i(M)$ is the number of $\leftmoon$.
\end{mycor}

\begin{proof}
  Let $M\in\cM(\lambda)$ and $i \in I$. By Lemma \ref{lem_eifiswitch}, if $\sigma^\red_i(M) = \rightmoon^a  \; \leftmoon^b$, then
  \begin{gather*}
    \sigma^\red_i(\tf_iM) = \rightmoon^{a+1} \; \leftmoon^{b-1} \text{ if } \varphi_i \left( M \right) > 0, \quad \text{and} \\
    \sigma^\red_i(\te_iM) = \rightmoon^{a-1} \; \leftmoon^{b+1} \text{ if } \varepsilon_i \left( M \right) > 0.
  \end{gather*}
  The result follows.
\end{proof}

We can now prove that the definitions of this section endow $\cM(\lambda)$ with the structure of a crystal.

\begin{myprop} \label{prop_crystal}
  Fix a partition $\lambda$.  The set of ice models $\cM(\lambda)$, together with the maps $\te_i$, $\tf_i$, $\wt$, $\varepsilon_i$, $\varphi_i$, $i \in I$, of Definitions~\ref{def_fiei} and~\ref{def_wt} are a crystal in the sense of Definition~\ref{crystal}.
\end{myprop}

\begin{proof}
  We verify the conditions of Definition~\ref{crystal}.

  \medskip

  \noindent \pref{C1}:  If $k$ is the number of pairs $\leftmoon \rightmoon$ that are cancelled in obtaining $\sigma^\red_i(M)$ from $\sigma^\red_i(M)$, we have
  \begin{multline*}
    \langle h_i,\wt(M)\rangle
    = \left\langle h_i,\sum_{j=1}^{n-1} a^M_j\epsilon_j \right\rangle
    = \sum_{j=1}^{n-1} a^M_j \left\langle h_i,\epsilon_j\right\rangle
    = \sum_{j=1}^{n-1} a^M_j(\delta_{i,j} - \delta_{i+1,j}) \\
    = a^M_i - a^M_{i+1}
    = \left( a^M_i - k \right) - \left( a^M_{i+1} - k \right)
    = \varphi_i(M) - \varepsilon_i(M),
  \end{multline*}
  where the final equality follows from Corollary~\ref{cor_defepsilonphi}.

  \medskip

  \noindent \pref{C2} \& (\ref{C3}):  These properties follow immediately from Lemma~\ref{lem_boxmoved} and the fact that $\alpha_i = \epsilon_i - \epsilon_{i+1}$.

  \medskip

  \noindent \pref{C4} \& \pref{C5}: These properties follow directly from Corollary~\ref{cor_defepsilonphi} and Lemma~\ref{lem_eifiswitch}.

  \medskip

  \noindent \pref{C6}: This follows immediately from the fact that the local changes of Figure~\ref{fig_localfi} are mutually inverse operations.
\end{proof}

We call $M(\lambda)$ the \emph{ice crystal} corresponding to the partition $\lambda$.

%
\section{Regularity of ice crystals} \label{sec_Reg}
%

The notion of a regular crystal was introduced by Stembridge in \cite{Stembridge}.  In this section we recall this definition in the special case of type $A$ (i.e.\ for $\fsl_n$), adjusting the notation slightly to match that of the current paper.  We will then show that the crystal $\cM(\lambda)$ defined in Section~\ref{sec_CrystalM} is regular.  This will be a key ingredient in the proof of our main result in Section~\ref{sec_CrystalM}.

Suppose we have an edge-coloured directed graph with underlying vertex set $\cB$, whose edges are coloured by elements from the set $I$.  Abusing notation, we will also let $\cB$ denote the directed graph.  We define $\te_i(b) = b'$ if there exists an $i$-coloured edge $b\leftarrow b'$, and dually define $\tf_i(b) = b'$ if there is an $i$-coloured edge $b\rightarrow b'$.  We define $\varepsilon_i$ and $\varphi_i$ as in \eqref{eq:varepsilon_i-def} and \eqref{eq:varphi_i-def}.  For $b \in \cB$, we define
\begin{gather*}
  \Delta_i\varphi_j(b) = \varphi_j(\te_i (b)) - \varphi_j(b), \qquad
  \Delta_i\varepsilon_j(b) = \varepsilon_j(b) - \varepsilon_j(\te_i(b)), \\
  \nabla_i\varphi_j(b) = \varphi_j(b) - \varphi_j(\tf_i (b)), \qquad
  \nabla_i\varepsilon_j(b) = \varepsilon_j(\tf_i(b)) - \varepsilon_j(b).
\end{gather*}

\begin{mydef}[{\cite[Def.~1.1]{Stembridge}}] \label{def:regular}
  We say $\cB$ is \emph{regular} if it satisfies the following properties.
  \begin{description}[style=multiline, labelwidth=1cm]
  	\item[\namedlabel{R1}{R1}] All monochromatic directed paths in $\cB$ have finite length.

    \item[\namedlabel{R2}{R2}] For all $b\in\cB$ and $i \in I$, there exists at most one $i$-coloured edge $b\to b'$ and $b \leftarrow b''$.

    \item[\namedlabel{R3}{R3}] If $\te_i(b)\neq 0$ and $i\neq j$, then $\Delta_i\varphi_j(b) + \Delta_i\varepsilon_j(b) = - \delta_{i+1,j} - \delta_{i,j+1}$.

    \item[\namedlabel{R4}{R4}] If $\te_i(b)\neq 0$ and $i\neq j$, then $\Delta_i\varphi_j(b),\Delta_i\varepsilon_j(b)\leq 0$.

    \item[\namedlabel{R5}{R5}] When $\te_i (b),\te_j (b) \neq 0$, then $\Delta_i\varepsilon_j(b) = 0$ implies $\te_i\te_j b=\te_j\te_ib$ and $\nabla_j\varphi_i(y) = 0$, where $y = \te_i\te_jb = \te_j\te_ib$.

    \item[\namedlabel{R6}{R6}] When $\te_i (b) ,\te_j (b) \neq 0$, then $\Delta_i\varepsilon_j(b) = \Delta_j\varepsilon_i(b) = -1$ implies $\te_i\te_j^2\te_ib = \te_j\te_i^2\te_jb$ and $\nabla_i\varphi_j(y)=\nabla_j\varphi_i(y) = -1$, where $y = \te_i\te_j^2\te_ib = \te_j\te_i^2\te_jb$.

    \item[\namedlabel{R5'}{R5$'$}] When $\tf_i (b) , \tf_j (b) \neq 0$, then $\Delta_i\varepsilon_j(b) = 0$ implies $\tf_i\tf_jb=\tf_j\tf_ib$ and $\nabla_j\varphi_i(y) = 0$, where $y = \tf_i\tf_jb = \tf_j\tf_ib$.

    \item[\namedlabel{R6'}{R6$'$}] When $\tf_i (b) ,\tf_j (b) \neq 0$, then $\Delta_i\varepsilon_j(b) = \Delta_j\varepsilon_i(b) = -1$ implies $\tf_i\tf_j^2\tf_ib = \tf_j\tf_i^2\tf_jb$ and $\nabla_i\varphi_j(y)=\nabla_j\varphi_i(y) = -1$, where $y = \tf_i\tf_j^2\tf_ib = \tf_j\tf_i^2\tf_jb$.
  \end{description}
\end{mydef}

Recall that $\cM(\lambda)$ is a crystal by Proposition~\ref{prop_crystal}.  As described above, we have the associated edge-coloured directed graph, with edges coloured by elements of $I$.

\begin{myprop} \label{prop_regular}
  The crystal $\cM(\lambda)$ is regular.
\end{myprop}

\begin{proof}
  We verify the conditions of Definition~\ref{def:regular}.

  \medskip

  \noindent \emph{Property} \pref{R1}:  This follows immediately from Remark~\ref{rem:finite-crystal-operator}.

  \medskip

  \noindent\emph{Property} \pref{R2}:  This follows directly from functionality of $\te_i$ and $\tf_i$.

  \medskip

  \noindent\emph{Property} \pref{R3}:  Suppose $\te_i(M) \ne 0$ and $i \ne j$.  Then
  \begin{multline*}
    \Delta_i \varphi_j(M) + \Delta_i \varepsilon_j(M)
    = \big( \varphi_j (\te_i M) - \varepsilon (\te_i M) \big) - \big( \varphi_j(M) - \varepsilon_j(M) \big) \\
    = \langle h_j, \wt (\te_i M) - \wt(M) \rangle
    = \langle h_j, \alpha_i \rangle
    = -\delta_{i+1,j} - \delta_{i,j+1},
  \end{multline*}
  where the last equality follows from \eqref{eq:Cartan-matrix-entries}.

  \medskip

  \noindent\emph{Property} \pref{R4}:  Suppose $\te_i(M) \neq 0$ and $i \ne j$.  Consider the following three cases:
  \begin{itemize}
    \item If $|i-j| > 1$, then $\sigma_j(M) = \sigma_j(\te_i M)$, and so $\Delta_i \varphi_j(M) = \Delta_i \varepsilon_j(M) = 0$.  Hence \pref{R4} is satisfied.

    \item If $j=i+1$, then the application of $\te_i$ to $M$ moves a box in row $i+1=j$ to a box in row $i=j-1$.  Thus, $\sigma_j(\te_i M)$ is obtained from $\sigma_j(M)$ by removing a $\leftmoon$.  Then \pref{R4} follows from Corollary~\ref{cor_defepsilonphi}.

    \item If $j=i-1$, then the application of $\te_i$ to $M$ moves a box in row $i+1=j+2$ to a box in row $i=j+1$.  Thus, $\sigma_j(\te_i M)$ is obtained from $\sigma_j(M)$ by adding a $\rightmoon$.  Again, \pref{R4} follows from Corollary~\ref{cor_defepsilonphi}.
  \end{itemize}

  \medskip

  \noindent\emph{Property} \pref{R5}:  Suppose $\te_i(M), \te_j(M)\neq 0$ and that $\Delta_i\varepsilon_j(M)=0$. If $i=j$, then \pref{R5} holds, being vacuously true. Similarly to the proof of \pref{R3} and \pref{R4}, we consider three cases for $i\neq j$:
  \begin{enumerate}[(i)]
    \item	$|i-j|>1$,	\label{r5i}
    \item	$j=i+1$,	\label{r5ii}
    \item	$j=i-1$.	\label{r5iii}
  \end{enumerate}

  \smallskip

  \noindent \emph{Case} \ref{r5i}:  We have already seen that $\Delta_i\varepsilon_j(M) = 0$. Because the rows of $M$ involved in the computation of the $i$ and $j$-signatures are disjoint here, we have that $\te_i\te_jM = \te_j\te_iM$. We then see that:
  \begin{multline*}
    \nabla_j\varphi_i(y) = \varphi_i(y) - \varphi_i(\tf_jy) = \varphi_i(\te_i\te_jM) - \varphi_i(\tf_j\te_j\te_iM) = \\
    = \varphi_i(\te_jM) + 1 - \varphi_i(\te_iM) = \varphi_i(\te_jM) - \varphi_i(M) = \Delta_j\varphi_i(M)= 0,
  \end{multline*}
  where the last equality follows from the fact that the rows involved in the computation of the $i$- and $j$-signatures are disjoint. Hence \pref{R5} holds for \ref{r5i}.

  \smallskip

  \noindent \emph{Case} \ref{r5ii}:  Assume $j=i+1$.  Suppose $u\in\beta(M)$ in row $j=i+1$ is moved by $\te_i$, and $v\in\beta(M)$ in row $j+1$ is moved by $\te_j$. Because $\Delta_i\varepsilon_j(M) = 0$, we know the number of $j$-movable boxes in row $j+1$ of $M$ is the same as that of $\te_i(M)$, after $u$ is moved. Considering the possibilities for $u$ which may or may not cancel in $\sigma^\red_j(M)$, we see $\Delta_i\varepsilon_j(M)=0$ if and only if:
  \begin{enumerate}[(a)]
    \item	the box $u$ is movable with respect to $\sigma_j(M)$; or \label{r5a}
    \item	the box $u$ is not movable with respect to $\sigma_j(M)$, its corresponding $\leftmoon$ in $\sigma^\red_j(M)$ cancelling with a $\rightmoon$ corresponding to a $w\in\beta(M)$ which further cancels with a $\leftmoon$ corresponding to an $x\in\beta(\te_iM)$ after moving of $u$. \label{r5b}
  \end{enumerate}
  For \ref{r5a}, we consider the relative position of $u$ and $v$. By assumption, both $u$ and $v$ are $j$-movable and hence the $\leftmoon$ of $u$ cannot appear before the $\rightmoon$ of $v$ in $\sigma^\red_j(M)$. So $u$ must be to the right of $v$. This scenario is depicted in Figure \ref{fig_r5_1}.

  \begin{figure}[!h]
    \begin{center}
      \begin{tikzpicture}
        \matrix {
        \node{$j+1$};	& \node{$\boxed{v}$}; \\
        \node{$j=i+1$};	& & \node{$\cdots$}; & \node(U){$\boxed{u}$};\\
        \node{$i$};\\
        };
      \end{tikzpicture}
      \caption{\label{fig_r5_1} The movement of $v$ will not impede the movement of $u$ and vice versa.}
    \end{center}
  \end{figure}
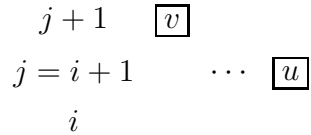

  Because $u$ is to the right of $v$, the movement of $v$ does not impede the movement of $u$ and vice versa, and thus $\te_i$ and $\te_j$ commute.

  For \ref{r5b}, the $\leftmoon$ corresponding to $u$ cancels with the $\rightmoon$ of $w$ which, subsequent to moving $u$, cancels with the $\leftmoon$ of some $x \in \beta(\te_iM)$. Note that $x$ necessarily is also a $j$-movable box.  Again, we consider the relative positions of $u$ and $v$. Note that $x$ must be to the right of $v$ since they are both movable, and hence since $x$ comes before $u$.  In addition, $u$ and $w$ are to the left of $v$.  This scenario is depicted in Figure \ref{fig_r5_2}.

  \begin{figure}[!h]
    \begin{center}
      \begin{tikzpicture}
        \matrix {
        \node{$j+1$};	& \node{$\boxed{v}$}; &&&& \coordinate(B); & \node{$\phantom{\boxed{v}}$}; & \node(W){$\boxed{w}$}; \\
        \node{$j=i+1$};	& & \node{$\cdots$}; & \node(X){$\boxed{x}$}; & \node{$\cdots$}; & \node(U){$\boxed{u}$}; & & \coordinate(A); \\
        \node{$i$};\\
        };
        \draw (U) -- (A) -- (W);
        \draw (U) -- (B) -- (W);
      \end{tikzpicture}
      \caption{\label{fig_r5_2} When $u$ moves, the $\rightmoon$ of $w$ cancels with the $\leftmoon$ of $x$.}
    \end{center}
  \end{figure}

  With this local configuration it is easy to see that $\te_i\te_j(M) = \te_j\te_i(M)$.  In both \ref{r5a} and \ref{r5b}, since $\te_i\te_j(M) = \te_j\te_i(M) = y$, we get the same result from \ref{r5i} that $\nabla_j\varphi_i(y) = \Delta_j\varphi_i(M)$, and so we now seek to show that the number of $\leftmoon$ in $\sigma^\red_i$ is unchanged in going from $M$ to $\te_j(M)$.

  In either \ref{r5a} or \ref{r5b}, suppose for a contradiction that the $\leftmoon$ corresponding to a box $z$ in the $i^\text{th}$ row of $M$ is canceled by the movement of $v$ to, say, $v'$ in the $(i+1)^\text{th}$ row of $\te_j(M)$. Then by the rules for cancellation, it must be below or to the left of $v$. This places $z$ to the left of $u$ in the $(i+1)^\text{th}$ row which is already assumed to be movable, meaning that the $\leftmoon$ corresponding to $z$ would appear to the left of the $\rightmoon$ of $u$ in $\sigma^\red_j(M)$, and hence they would cancel, resulting in a contradiction. Since the movement of $v$ to $v'$ does not move a box into row $i$, it is clear that $\varphi_i$ is unchanged in going from $M$ to $\te_j(M)$, and thus $\nabla_j\varphi_i(y) = \Delta_j\varphi_i(M) = 0$, as desired.

  \smallskip

  \noindent \emph{Case } \ref{r5iii}:  This case is similar and its proof will be omitted.

  \medskip

  \noindent\emph{Proof of} \pref{R6}:
  Suppose that $\te_i(M),\te_j(M) \neq 0$. We note that, by the above arguments, $\Delta_i\varepsilon_i(M) = -1$ occurs only in two cases:
  \begin{enumerate}[(i)]
    \item	$j = i+1$,	\label{r6i}
    \item	$j = i-1$.	\label{r6ii}
  \end{enumerate}
  We consider two boxes: $u \in \beta(M)$ which moves after application of $\te_j$ to $u'\in\beta(\te_jM)$, and $v\in\beta(M)$ which moves after application of $\te_i$ to $v'\in\beta(\te_iM)$.

  First consider case \ref{r6i}.   We have $\Delta_i\varepsilon_j(M) = \Delta_j\varepsilon_i(M) = -1$ if and only if the $\rightmoon$ corresponding to $u'$ does not cancel in $\sigma^\red_i(\te_iM)$ and the $\leftmoon$ corresponding to $v$ cancels in $\sigma^\red_j(M)$. Locally, $M$ must appear as in Figure \ref{fig_property6a} where the $\rightmoon$ corresponding to the box $w$ cancels with the $\leftmoon$ of $v$ in $\sigma^\red_j(M)$. We now show that $\te_i\te_j^2\te_iM = \te_j\te_i^2\te_jM$.

  \begin{figure}[!h]
    \begin{center}
      \begin{tikzpicture}
        \draw [dashed, thin] (-1,-1) -- (0,-1);
        \draw [dashed, thin] (-3,0) -- (0,0);
        \draw [dashed, thin] (-3,1) -- (0,1);
        \draw [dashed, thin] (1,-1) -- (3,-1);
        \draw [dashed, thin] (1,0) -- (3,0);
        \draw [dashed, thin] (4,0) -- (6,0);
        \draw [dashed, thin] (4,1) -- (5,1);
        \draw [dashed, thin] (-2.5,-0.5) -- (-2.5,1.5);
        \draw [dashed, thin] (-1.5,-0.5) -- (-1.5,1.5);
        \draw [dashed, thin] (-0.5,-1.5) -- (-0.5,1.5);
        \draw [dashed, thin] (1.5,-1.5) -- (1.5,0.5);
        \draw [dashed, thin] (2.5,-1.5) -- (2.5,0.5);
        \draw [dashed, thin] (4.5,-0.5) -- (4.5,1.5);
        \draw [dashed, thin] (5.5,-0.5) -- (5.5,0.5);
        \node at (-4.5,-1) {$i$};
        \node at (-4.5,0) {$i+1$};
        \node at (-4.5,1) {$i+2$};
        \draw [fill=white,ultra thin] (-3,0) circle [radius=0.2];
        \node at (-3,0) {$ \pm $};
        \node[left] at (-3.2,0) {$\cdots$};
        \draw [fill=white,ultra thin] (-3,1) circle [radius=0.2];
        \node at (-3,1) {$ - $};
        \draw [fill=white,ultra thin] (-2,0) circle [radius=0.2];
        \node at (-2,0) {$ + $};
        \draw [fill=white,ultra thin] (-2,1) circle [radius=0.2];
        \node at (-2,1) {$ - $};
        \draw [fill=white,ultra thin] (-1,-1) circle [radius=0.2];
        \node at (-1,-1) {$ \pm $};
        \draw [fill=white,ultra thin] (-1,0) circle [radius=0.2];
        \node at (-1,0) {$ - $};
        \draw [fill=white,ultra thin] (-1,1) circle [radius=0.2];
        \node at (-1,1) {$ \pm $};
        \draw [fill=white,ultra thin] (0,-1) circle [radius=0.2];
        \node at (0,-1) {$ \pm $};
        \draw [fill=white,ultra thin] (0,0) circle [radius=0.2];
        \node at (0,0) {$ \pm $};
        \draw [fill=white,ultra thin] (0,1) circle [radius=0.2];
        \node at (0,1) {$ \pm $};
        \draw [fill=white,ultra thin] (1,-1) circle [radius=0.2];
        \node at (1,-1) {$ \pm $};
        \node[left] at (0.9,-1) {$\cdots$};
        \draw [fill=white,ultra thin] (1,0) circle [radius=0.2];
        \node at (1,0) {$ - $};
        \node[left] at (0.9,0) {$\cdots$};
        \draw [fill=white,ultra thin] (2,-1) circle [radius=0.2];
        \node at (2,-1) {$ + $};
        \draw [fill=white,ultra thin] (2,0) circle [radius=0.2];
        \node at (2,0) {$ - $};
        \draw [fill=white,ultra thin] (3,-1) circle [radius=0.2];
        \node at (3,-1) {$ - $};
        \draw [fill=white,ultra thin] (3,0) circle [radius=0.2];
        \node at (3,0) {$ \pm $};
        \draw [fill=white,ultra thin] (4,0) circle [radius=0.2];
        \node at (4,0) {$ \pm $};
        \node[left] at (3.9,0) {$\cdots$};
        \draw [fill=white,ultra thin] (4,1) circle [radius=0.2];
        \node at (4,1) {$ - $};
        \draw [fill=white,ultra thin] (5,0) circle [radius=0.2];
        \node at (5,0) {$ \pm $};
        \draw [fill=white,ultra thin] (5,1) circle [radius=0.2];
        \node at (5,1) {$ - $};
        \draw [fill=white,ultra thin] (6,0) circle [radius=0.2];
        \node at (6,0) {$ \pm $};
        \node[right] at (6.1,0) {$\cdots$};
        \draw [fill=white,ultra thin] (-2.5,-0.5) circle [radius=0.2];
        \node at (-2.5,-0.5) {$ \pm $};
        \draw [fill=white,ultra thin] (-2.5,0.5) circle [radius=0.2];
        \node at (-2.5,0.5) {$ + $};
        \draw [fill=white,ultra thin] (-2.5,1.5) circle [radius=0.2];
        \node at (-2.5,1.5) {$ + $};
        \draw [fill=white,ultra thin] (-1.5,-0.5) circle [radius=0.2];
        \node at (-1.5,-0.5) {$ + $};
        \draw [fill=white,ultra thin] (-1.5,0.5) circle [radius=0.2];
        \node at (-1.5,0.5) {$ - $};
        \draw [fill=white,ultra thin] (-1.5,1.5) circle [radius=0.2];
        \node at (-1.5,1.5) {$ \pm $};
        \draw [fill=white,ultra thin] (-0.5,-1.5) circle [radius=0.2];
        \node at (-0.5,-1.5) {$ \pm $};
        \draw [fill=white,ultra thin] (-0.5,-0.5) circle [radius=0.2];
        \node at (-0.5,-0.5) {$ \pm $};
        \draw [fill=white,ultra thin] (-0.5,0.5) circle [radius=0.2];
        \node at (-0.5,0.5) {$ \pm $};
        \draw [fill=white,ultra thin] (-0.5,1.5) circle [radius=0.2];
        \node at (-0.5,1.5) {$ \pm $};
        \draw [fill=white,ultra thin] (1.5,-1.5) circle [radius=0.2];
        \node at (1.5,-1.5) {$ \pm $};
        \draw [fill=white,ultra thin] (1.5,-0.5) circle [radius=0.2];
        \node at (1.5,-0.5) {$ + $};
        \draw [fill=white,ultra thin] (1.5,0.5) circle [radius=0.2];
        \node at (1.5,0.5) {$ + $};
        \draw [fill=white,ultra thin] (2.5,-1.5) circle [radius=0.2];
        \node at (2.5,-1.5) {$ \pm $};
        \draw [fill=white,ultra thin] (2.5,-0.5) circle [radius=0.2];
        \node at (2.5,-0.5) {$ - $};
        \draw [fill=white,ultra thin] (2.5,0.5) circle [radius=0.2];
        \node at (2.5,0.5) {$ \pm $};
        \draw [fill=white,ultra thin] (4.5,-0.5) circle [radius=0.2];
        \node at (4.5,-0.5) {$ \pm $};
        \draw [fill=white,ultra thin] (4.5,0.5) circle [radius=0.2];
        \node at (4.5,0.5) {$ + $};
        \draw [fill=white,ultra thin] (4.5,1.5) circle [radius=0.2];
        \node at (4.5,1.5) {$ + $};
        \draw [fill=white,ultra thin] (5.5,-0.5) circle [radius=0.2];
        \node at (5.5,-0.5) {$ + $};
        \draw [fill=white,ultra thin] (5.5,0.5) circle [radius=0.2];
        \node at (5.5,0.5) {$ - $};
        \node at (-2.5,1) [rectangle,draw,fill=white] {$ u $};
        \node at (1.5,0) [rectangle,draw,fill=white] {$ v $};
        \node at (4.5,1) [rectangle,draw,fill=white] {$ w $};
        \draw [fill=gray] (-2.5,0) circle[radius=0.1];
        \draw [fill=gray] (-1.5,0) circle[radius=0.1];
        \draw [fill=gray] (-1.5,1) circle[radius=0.1];
        \draw [fill=gray] (-0.5,-1) circle[radius=0.1];
        \draw [fill=gray] (-0.5,0) circle[radius=0.1];
        \draw [fill=gray] (-0.5,1) circle[radius=0.1];
        \draw [fill=gray] (1.5,-1) circle[radius=0.1];
        \draw [fill=gray] (2.5,-1) circle[radius=0.1];
        \draw [fill=gray] (2.5,0) circle[radius=0.1];
        \draw [fill=gray] (4.5,0) circle[radius=0.1];
        \draw [fill=gray] (5.5,0) circle[radius=0.1];
  		\end{tikzpicture}
      \caption{Local configuration of $M$.\label{fig_property6a}}
    \end{center}
  \end{figure}

  Application of $\te_i$ to $M$ moves $v$ diagonally to $v'$ in row $i$. Application of $\te_j$ to $\te_i(M)$ then moves $w$ to $w'$ in row $j$. We are able to do this since its corresponding $\rightmoon$ appears in $\sigma^\red_j(\te_iM)$ now that $v$ is no longer present. No movable boxes exist between $u$ and $w$ in $M$ because $u$ was the leftmost movable box, and thus application of $\te_j$ again yields that $u$ moves to $u'$ in row $j$.

  Since $v$ was the box to move after application of $\te_i$, we know that there are no movable boxes in row $j$ beyond $v'$, besides possibly $w'$. There is a possibility that there exists a movable box $x$ between $v'$ and $w'$, but the $\rightmoon$ of $w'$ is cancelled by at least the $\leftmoon$ of $v'$ in $\sigma^\red_i(\te_j^2e_iM)$, if not by the $\leftmoon$ of the possible $x$.

  Because the $\rightmoon$ of $v$ was not cancelled previously, we know there are no movable boxes in row $i$ between $u'$ and $v'$, as well as by assumption that $u'$ does not cancel in $\sigma^\red_i(\te_jM)$, still holding for $\sigma^\red_i(\te_j^2e_iM)$ since no boxes have newly moved to a position to the left of $u'$ in row $i$.

  Thus $\te_i\te_j^2\te_iM\neq 0$ since at least $u'$ is movable. Application of $\te_i$ to $\te_j^2\te_iM$ will move some movable box in row $j$ between $u'$ and where $v$ used to be, including $u'$ as a possibility. Without loss of generality, we suppose $u'$ moves again---supposing there are no other boxes in this given area---to a box which we will call $u''$. We call $y = \te_i\te_j^2\te_iM$ which locally appears as in Figure \ref{fig_property6aa}.

  \begin{figure}[!h]
    \begin{center}
      \begin{tikzpicture}
        \draw [dashed, thin] (-1,-1) -- (0,-1);
        \draw [dashed, thin] (-3,0) -- (0,0);
        \draw [dashed, thin] (-3,1) -- (0,1);
        \draw [dashed, thin] (1,-1) -- (3,-1);
        \draw [dashed, thin] (1,0) -- (3,0);
        \draw [dashed, thin] (4,0) -- (6,0);
        \draw [dashed, thin] (4,1) -- (5,1);
        \draw [dashed, thin] (-2.5,-0.5) -- (-2.5,1.5);
        \draw [dashed, thin] (-1.5,-0.5) -- (-1.5,1.5);
        \draw [dashed, thin] (-0.5,-1.5) -- (-0.5,1.5);
        \draw [dashed, thin] (1.5,-1.5) -- (1.5,0.5);
        \draw [dashed, thin] (2.5,-1.5) -- (2.5,0.5);
        \draw [dashed, thin] (4.5,-0.5) -- (4.5,1.5);
        \draw [dashed, thin] (5.5,-0.5) -- (5.5,0.5);
        \node at (-4.5,-1) {$i$};
        \node at (-4.5,0) {$j$};
        \node at (-4.5,1) {$j+1$};
        \draw [fill=white,ultra thin] (-3,0) circle [radius=0.2];
        \node at (-3,0) {$ \pm $};
        \draw [fill=white,ultra thin] (-3,1) circle [radius=0.2];
        \node at (-3,1) {$ - $};
        \draw [fill=white,ultra thin] (-2,0) circle [radius=0.2];
        \node at (-2,0) {$ - $};
        \draw [fill=white,ultra thin] (-2,1) circle [radius=0.2];
        \node at (-2,1) {$ + $};
        \draw [fill=white,ultra thin] (-1,-1) circle [radius=0.2];
        \node at (-1,-1) {$ - $};
        \draw [fill=white,ultra thin] (-1,0) circle [radius=0.2];
        \node at (-1,0) {$ + $};
        \draw [fill=white,ultra thin] (-1,1) circle [radius=0.2];
        \node at (-1,1) {$ \pm $};
        \draw [fill=white,ultra thin] (0,-1) circle [radius=0.2];
        \node at (0,-1) {$ - $};
        \draw [fill=white,ultra thin] (0,0) circle [radius=0.2];
        \node at (0,0) {$ \pm $};
        \draw [fill=white,ultra thin] (0,1) circle [radius=0.2];
        \node at (0,1) {$ \pm $};
        \draw [fill=white,ultra thin] (1,-1) circle [radius=0.2];
        \node at (1,-1) {$ \pm $};
        \node[left] at (0.9,-1) {$\cdots$};
        \draw [fill=white,ultra thin] (1,0) circle [radius=0.2];
        \node at (1,0) {$ - $};
        \node[left] at (0.9,0) {$\cdots$};
        \draw [fill=white,ultra thin] (2,-1) circle [radius=0.2];
        \node at (2,-1) {$ - $};
        \draw [fill=white,ultra thin] (2,0) circle [radius=0.2];
        \node at (2,0) {$ + $};
        \draw [fill=white,ultra thin] (3,-1) circle [radius=0.2];
        \node at (3,-1) {$ - $};
        \draw [fill=white,ultra thin] (3,0) circle [radius=0.2];
        \node at (3,0) {$ \pm $};
        \draw [fill=white,ultra thin] (4,0) circle [radius=0.2];
        \node at (4,0) {$ \pm $};
        \node[left] at (3.9,0) {$\cdots$};
        \draw [fill=white,ultra thin] (4,1) circle [radius=0.2];
        \node at (4,1) {$ - $};
        \draw [fill=white,ultra thin] (5,0) circle [radius=0.2];
        \node at (5,0) {$ - $};
        \draw [fill=white,ultra thin] (5,1) circle [radius=0.2];
        \node at (5,1) {$ + $};
        \draw [fill=white,ultra thin] (6,0) circle [radius=0.2];
        \node at (6,0) {$ - $};
        \node[right] at (6.1,0) {$\cdots$};
        \draw [fill=white,ultra thin] (-2.5,-0.5) circle [radius=0.2];
        \node at (-2.5,-0.5) {$ \pm $};
        \draw [fill=white,ultra thin] (-2.5,0.5) circle [radius=0.2];
        \node at (-2.5,0.5) {$ - $};
        \draw [fill=white,ultra thin] (-2.5,1.5) circle [radius=0.2];
        \node at (-2.5,1.5) {$ + $};
        \draw [fill=white,ultra thin] (-1.5,-0.5) circle [radius=0.2];
        \node at (-1.5,-0.5) {$ - $};
        \draw [fill=white,ultra thin] (-1.5,0.5) circle [radius=0.2];
        \node at (-1.5,0.5) {$ + $};
        \draw [fill=white,ultra thin] (-1.5,1.5) circle [radius=0.2];
        \node at (-1.5,1.5) {$ \pm $};
        \draw [fill=white,ultra thin] (-0.5,-1.5) circle [radius=0.2];
        \node at (-0.5,-1.5) {$ + $};
        \draw [fill=white,ultra thin] (-0.5,-0.5) circle [radius=0.2];
        \node at (-0.5,-0.5) {$ + $};
        \draw [fill=white,ultra thin] (-0.5,0.5) circle [radius=0.2];
        \node at (-0.5,0.5) {$ \pm $};
        \draw [fill=white,ultra thin] (-0.5,1.5) circle [radius=0.2];
        \node at (-0.5,1.5) {$ \pm $};
        \draw [fill=white,ultra thin] (1.5,-1.5) circle [radius=0.2];
        \node at (1.5,-1.5) {$ \pm $};
        \draw [fill=white,ultra thin] (1.5,-0.5) circle [radius=0.2];
        \node at (1.5,-0.5) {$ - $};
        \draw [fill=white,ultra thin] (1.5,0.5) circle [radius=0.2];
        \node at (1.5,0.5) {$ + $};
        \draw [fill=white,ultra thin] (2.5,-1.5) circle [radius=0.2];
        \node at (2.5,-1.5) {$ + $};
        \draw [fill=white,ultra thin] (2.5,-0.5) circle [radius=0.2];
        \node at (2.5,-0.5) {$ + $};
        \draw [fill=white,ultra thin] (2.5,0.5) circle [radius=0.2];
        \node at (2.5,0.5) {$ \pm $};
        \draw [fill=white,ultra thin] (4.5,-0.5) circle [radius=0.2];
        \node at (4.5,-0.5) {$ \pm $};
        \draw [fill=white,ultra thin] (4.5,0.5) circle [radius=0.2];
        \node at (4.5,0.5) {$ - $};
        \draw [fill=white,ultra thin] (4.5,1.5) circle [radius=0.2];
        \node at (4.5,1.5) {$ + $};
        \draw [fill=white,ultra thin] (5.5,-0.5) circle [radius=0.2];
        \node at (5.5,-0.5) {$ + $};
        \draw [fill=white,ultra thin] (5.5,0.5) circle [radius=0.2];
        \node at (5.5,0.5) {$ + $};
        \node at (-0.5,-1) [rectangle,draw,fill=white] {$ u'' $};
        \node at (2.5,-1) [rectangle,draw,fill=white] {$ v' $};
        \node at (5.5,0) [rectangle,draw,fill=white] {$ w' $};
        \draw [fill=gray] (-2.5,0) circle[radius=0.1];
        \draw [fill=gray] (-2.5,1) circle[radius=0.1];
        \draw [fill=gray] (-1.5,0) circle[radius=0.1];
        \draw [fill=gray] (-1.5,1) circle[radius=0.1];
        \draw [fill=gray] (-0.5,0) circle[radius=0.1];
        \draw [fill=gray] (-0.5,1) circle[radius=0.1];
        \draw [fill=gray] (1.5,-1) circle[radius=0.1];
        \draw [fill=gray] (1.5,0) circle[radius=0.1];
        \draw [fill=gray] (2.5,0) circle[radius=0.1];
        \draw [fill=gray] (4.5,0) circle[radius=0.1];
        \draw [fill=gray] (4.5,1) circle[radius=0.1];
  		\end{tikzpicture}
      \caption{Local configuration of $y = \te_i\te_j^2\te_iM$.\label{fig_property6aa}}
    \end{center}
  \end{figure}

  We now consider the local configuration of $\te_j\te_i^2\te_jM$. Application of $\te_j$ to $M$ results in the movement of $u$ to $u'$. Subsequent application of $\te_i$ results in the movement of $u'$ to $u''$, still supposing (without loss of generality) that there are no boxes between $u'$ and $v$ in row $j$. Application of $\te_i$ again results in the movement of $v$ to $v'$, and now since the $\rightmoon$ of $w$ is not cancelled by the $\leftmoon$ of $v$ any longer, a subsequent application of $\te_j$ moves $w$ to $w'$. The result is precisely $y$ since no other boxes were moved, the changes only being local changes in the ice model. Hence we have $\te_i\te_j^2\te_iM = \te_j\te_i^2\te_jM = y$.

  We have:
  \begin{multline*}
    \nabla_i\varphi_j(y) = \varphi_j(y) - \varphi_j(\tf_iy) = \varphi_j(\te_j\te_i^2\te_jM) - \varphi_j(\tf_i\te_i\te_j^2\te_iM) =\\
    = \varphi_j(\te_i^2\te_jM) + 1 - \varphi_j(\te_j^2\te_iM) = \varphi_j(\te_i^2\te_jM) - \varphi_j(\te_iM) - 1,
  \end{multline*}
  and similarly,
  \[
    \nabla_j\varphi_i(y) = \varphi_i(\te_j^2\te_iM) - \varphi_i(\te_jM) - 1.
  \]
  Hence we wish to show that $\varphi_j(\te_i^2\te_jM) = \varphi_j(\te_iM)$ and $\varphi_i(\te_j^2\te_iM) = \varphi_i(\te_jM)$.

  Considering the first equality, we consider only rows $j$ and $j+1$, where we check whether the number of $\leftmoon$'s remains unchanged in going from $\sigma^\red_j(\te_iM)$ to $\sigma^\red_j(\te_i^2\te_jM)$, the local configurations pictured as in Figures~\ref{eb} and \ref{eeb}.

  \begin{figure}[!h]
    \begin{center}
      \begin{tikzpicture}
        \draw [dashed, thin] (-3,0) -- (-1,0);
        \draw [dashed, thin] (-3,1) -- (-1,1);
        \draw [dashed, thin] (0,0) -- (1,0);
        \draw [dashed, thin] (2,0) -- (3,0);
        \draw [dashed, thin] (4,0) -- (6,0);
        \draw [dashed, thin] (4,1) -- (5,1);
        \draw [dashed, thin] (-2.5,-0.5) -- (-2.5,1.5);
        \draw [dashed, thin] (-1.5,-0.5) -- (-1.5,1.5);
        \draw [dashed, thin] (0.5,-0.5) -- (0.5,0.5);
        \draw [dashed, thin] (2.5,-0.5) -- (2.5,0.5);
        \draw [dashed, thin] (4.5,-0.5) -- (4.5,1.5);
        \draw [dashed, thin] (5.5,-0.5) -- (5.5,0.5);
        \node at (-4,0) {$j$};
        \node at (-4,1) {$j+1$};
        \draw [fill=white,ultra thin] (-3,0) circle [radius=0.2];
        \node at (-3,0) {$ \pm $};
        \draw [fill=white,ultra thin] (-3,1) circle [radius=0.2];
        \node at (-3,1) {$ - $};
        \draw [fill=white,ultra thin] (-2,0) circle [radius=0.2];
        \node at (-2,0) {$ + $};
        \draw [fill=white,ultra thin] (-2,1) circle [radius=0.2];
        \node at (-2,1) {$ - $};
        \draw [fill=white,ultra thin] (-1,0) circle [radius=0.2];
        \node at (-1,0) {$ - $};
        \draw [fill=white,ultra thin] (-1,1) circle [radius=0.2];
        \node at (-1,1) {$ \pm $};
        \node[right] at (-0.9,1) {$\cdots$};
        \draw [fill=white,ultra thin] (0,0) circle [radius=0.2];
        \node at (0,0) {$ - $};
        \node[left] at (-0.1,0) {$\cdots$};
        \draw [fill=white,ultra thin] (1,0) circle [radius=0.2];
        \node at (1,0) {$ - $};
        \draw [fill=white,ultra thin] (2,0) circle [radius=0.2];
        \node at (2,0) {$ - $};
        \node[left] at (1.9,0) {$\cdots$};
        \draw [fill=white,ultra thin] (3,0) circle [radius=0.2];
        \node at (3,0) {$ + $};
        \draw [fill=white,ultra thin] (4,0) circle [radius=0.2];
        \node at (4,0) {$ \pm $};
        \node[left] at (3.9,0) {$\cdots$};
        \draw [fill=white,ultra thin] (4,1) circle [radius=0.2];
        \node at (4,1) {$ - $};
        \node[left] at (3.9,1) {$\cdots$};
        \draw [fill=white,ultra thin] (5,0) circle [radius=0.2];
        \node at (5,0) {$ + $};
        \draw [fill=white,ultra thin] (5,1) circle [radius=0.2];
        \node at (5,1) {$ - $};
        \draw [fill=white,ultra thin] (6,0) circle [radius=0.2];
        \node at (6,0) {$ - $};
        \node[right] at (6.1,0) {$\cdots$};
        \draw [fill=white,ultra thin] (-2.5,-0.5) circle [radius=0.2];
        \node at (-2.5,-0.5) {$ \pm $};
        \draw [fill=white,ultra thin] (-2.5,0.5) circle [radius=0.2];
        \node at (-2.5,0.5) {$ + $};
        \draw [fill=white,ultra thin] (-2.5,1.5) circle [radius=0.2];
        \node at (-2.5,1.5) {$ + $};
        \draw [fill=white,ultra thin] (-1.5,-0.5) circle [radius=0.2];
        \node at (-1.5,-0.5) {$ + $};
        \draw [fill=white,ultra thin] (-1.5,0.5) circle [radius=0.2];
        \node at (-1.5,0.5) {$ - $};
        \draw [fill=white,ultra thin] (-1.5,1.5) circle [radius=0.2];
        \node at (-1.5,1.5) {$ \pm $};
        \draw [fill=white,ultra thin] (0.5,-0.5) circle [radius=0.2];
        \node at (0.5,-0.5) {$ + $};
        \draw [fill=white,ultra thin] (0.5,0.5) circle [radius=0.2];
        \node at (0.5,0.5) {$ + $};
        \draw [fill=white,ultra thin] (2.5,-0.5) circle [radius=0.2];
        \node at (2.5,-0.5) {$ - $};
        \draw [fill=white,ultra thin] (2.5,0.5) circle [radius=0.2];
        \node at (2.5,0.5) {$ + $};
        \draw [fill=white,ultra thin] (4.5,-0.5) circle [radius=0.2];
        \node at (4.5,-0.5) {$ \pm $};
        \draw [fill=white,ultra thin] (4.5,0.5) circle [radius=0.2];
        \node at (4.5,0.5) {$ + $};
        \draw [fill=white,ultra thin] (4.5,1.5) circle [radius=0.2];
        \node at (4.5,1.5) {$ + $};
        \draw [fill=white,ultra thin] (5.5,-0.5) circle [radius=0.2];
        \node at (5.5,-0.5) {$ + $};
        \draw [fill=white,ultra thin] (5.5,0.5) circle [radius=0.2];
        \node at (5.5,0.5) {$ - $};
        \node at (-2.5,1) [rectangle,draw,fill=white] {$ u $};
        \node at (0.5,0) [rectangle,draw,fill=white] {$ x $};
        \node at (4.5,1) [rectangle,draw,fill=white] {$ w $};
        \draw [fill=gray] (-2.5,0) circle[radius=0.1];
        \draw [fill=gray] (-1.5,0) circle[radius=0.1];
        \draw [fill=gray] (-1.5,1) circle[radius=0.1];
        \draw [fill=gray] (2.5,0) circle[radius=0.1];
        \draw [fill=gray] (4.5,0) circle[radius=0.1];
        \draw [fill=gray] (5.5,0) circle[radius=0.1];
      \end{tikzpicture}
      \caption{Local configuration of $\te_i\left( M \right)$.\label{eb}}
    \end{center}
  \end{figure}

  \begin{figure}[!h]
    \begin{center}
      \begin{tikzpicture}
        \draw [dashed, thin] (-3,0) -- (-1,0);
        \draw [dashed, thin] (-3,1) -- (-1,1);
        \draw [dashed, thin] (0,0) -- (1,0);
        \draw [dashed, thin] (2,0) -- (3,0);
        \draw [dashed, thin] (4,0) -- (6,0);
        \draw [dashed, thin] (4,1) -- (5,1);
        \draw [dashed, thin] (-2.5,-0.5) -- (-2.5,1.5);
        \draw [dashed, thin] (-1.5,-0.5) -- (-1.5,1.5);
        \draw [dashed, thin] (0.5,-0.5) -- (0.5,0.5);
        \draw [dashed, thin] (2.5,-0.5) -- (2.5,0.5);
        \draw [dashed, thin] (4.5,-0.5) -- (4.5,1.5);
        \draw [dashed, thin] (5.5,-0.5) -- (5.5,0.5);
        \node at (-4,0) {$j$};
        \node at (-4,1) {$j+1$};
        \draw [fill=white,ultra thin] (-3,0) circle [radius=0.2];
        \node at (-3,0) {$ \pm $};
        \draw [fill=white,ultra thin] (-3,1) circle [radius=0.2];
        \node at (-3,1) {$ - $};
        \draw [fill=white,ultra thin] (-2,0) circle [radius=0.2];
        \node at (-2,0) {$ - $};
        \draw [fill=white,ultra thin] (-2,1) circle [radius=0.2];
        \node at (-2,1) {$ + $};
        \draw [fill=white,ultra thin] (-1,0) circle [radius=0.2];
        \node at (-1,0) {$ - $};
        \draw [fill=white,ultra thin] (-1,1) circle [radius=0.2];
        \node at (-1,1) {$ \pm $};
        \node[right] at (-0.9,1) {$\cdots$};
        \draw [fill=white,ultra thin] (0,0) circle [radius=0.2];
        \node at (0,0) {$ - $};
        \node[left] at (-0.1,0) {$\cdots$};
        \draw [fill=white,ultra thin] (1,0) circle [radius=0.2];
        \node at (1,0) {$ + $};
        \draw [fill=white,ultra thin] (2,0) circle [radius=0.2];
        \node at (2,0) {$ - $};
        \node[left] at (1.9,0) {$\cdots$};
        \draw [fill=white,ultra thin] (3,0) circle [radius=0.2];
        \node at (3,0) {$ + $};
        \draw [fill=white,ultra thin] (4,0) circle [radius=0.2];
        \node at (4,0) {$ \pm $};
        \node[left] at (3.9,0) {$\cdots$};
        \draw [fill=white,ultra thin] (4,1) circle [radius=0.2];
        \node at (4,1) {$ - $};
        \node[left] at (3.9,1) {$\cdots$};
        \draw [fill=white,ultra thin] (5,0) circle [radius=0.2];
        \node at (5,0) {$ + $};
        \draw [fill=white,ultra thin] (5,1) circle [radius=0.2];
        \node at (5,1) {$ - $};
        \draw [fill=white,ultra thin] (6,0) circle [radius=0.2];
        \node at (6,0) {$ - $};
        \node[right] at (6.1,0) {$\cdots$};
        \draw [fill=white,ultra thin] (-2.5,-0.5) circle [radius=0.2];
        \node at (-2.5,-0.5) {$ \pm $};
        \draw [fill=white,ultra thin] (-2.5,0.5) circle [radius=0.2];
        \node at (-2.5,0.5) {$ - $};
        \draw [fill=white,ultra thin] (-2.5,1.5) circle [radius=0.2];
        \node at (-2.5,1.5) {$ + $};
        \draw [fill=white,ultra thin] (-1.5,-0.5) circle [radius=0.2];
        \node at (-1.5,-0.5) {$ \pm $};
        \draw [fill=white,ultra thin] (-1.5,0.5) circle [radius=0.2];
        \node at (-1.5,0.5) {$ + $};
        \draw [fill=white,ultra thin] (-1.5,1.5) circle [radius=0.2];
        \node at (-1.5,1.5) {$ \pm $};
        \draw [fill=white,ultra thin] (0.5,-0.5) circle [radius=0.2];
        \node at (0.5,-0.5) {$ - $};
        \draw [fill=white,ultra thin] (0.5,0.5) circle [radius=0.2];
        \node at (0.5,0.5) {$ + $};
        \draw [fill=white,ultra thin] (2.5,-0.5) circle [radius=0.2];
        \node at (2.5,-0.5) {$ - $};
        \draw [fill=white,ultra thin] (2.5,0.5) circle [radius=0.2];
        \node at (2.5,0.5) {$ + $};
        \draw [fill=white,ultra thin] (4.5,-0.5) circle [radius=0.2];
        \node at (4.5,-0.5) {$ \pm $};
        \draw [fill=white,ultra thin] (4.5,0.5) circle [radius=0.2];
        \node at (4.5,0.5) {$ + $};
        \draw [fill=white,ultra thin] (4.5,1.5) circle [radius=0.2];
        \node at (4.5,1.5) {$ + $};
        \draw [fill=white,ultra thin] (5.5,-0.5) circle [radius=0.2];
        \node at (5.5,-0.5) {$ + $};
        \draw [fill=white,ultra thin] (5.5,0.5) circle [radius=0.2];
        \node at (5.5,0.5) {$ - $};
        \node at (-1.5,0) [rectangle,draw,fill=white] {$ u' $};
        \node at (4.5,1) [rectangle,draw,fill=white] {$ w $};
        \draw [fill=gray] (-2.5,0) circle[radius=0.1];
        \draw [fill=gray] (-2.5,1) circle[radius=0.1];
        \draw [fill=gray] (-1.5,1) circle[radius=0.1];
        \draw [fill=gray] (0.5,0) circle[radius=0.1];
        \draw [fill=gray] (2.5,0) circle[radius=0.1];
        \draw [fill=gray] (4.5,0) circle[radius=0.1];
        \draw [fill=gray] (5.5,0) circle[radius=0.1];
  		\end{tikzpicture}
      \caption{Local configuration of $\te_i^2\te_jM$.\label{eeb}}
    \end{center}
  \end{figure}

  We draw attention to the addition of the box $x$ to the diagram, which is simply the box that is moved in $\te_i\te_jM$ by $\te_i$, which was treated without loss of generality as $u'$ when showing that $\te_i\te_j^2\te_iM =\te_j\te_i^2\te_jM$.

  We note in Figure~\ref{eb} that the $\leftmoon$ of $x$ cancels with the $\rightmoon$ of $w$ in $\sigma^\red_j(\te_iM)$ as there are no other movable boxes in row $j$ between $x$ and $w$. However, in Figure~\ref{eeb} we have that $x$ has moved from row $j$ to $i$ and thus is not in $\sigma_j(\te_i^2\te_jM)$. The only other differing box is the presence of $u'$ as a box in row $j$, however this adds another $\leftmoon$ to the $j$-signature, either taking place of $x$, its $\leftmoon$ cancelling with the $\rightmoon$ of $w$, or replacing the $\leftmoon$ contributed by whichever movable box between $u'$ and the former location of $x$, whose $\leftmoon$ now cancels with the $\rightmoon$ of $w$ in $\sigma^\red_j(\te_i^2\sigma_jM)$.

  Hence we have the equality, $\varphi_j(\te_i^2\te_jM) = \varphi_j(\te_iM)$. The other equality holds with similar reasoning. Hence we have in case \ref{r6i}, where $j=i+1$, that \pref{R6} holds.  Case \ref{r6ii}, where $j=i-1$, is proved in a similar manner, and hence \pref{R6} holds in general.

  \medskip

  \noindent \emph{Properties} \pref{R5'} \& \pref{R6'}:  The proofs of these properties are similar to those of \pref{R5} and \pref{R6}, and so will be omitted.
\end{proof}

%
\section{Main result} \label{sec_CrystalM}
%

In this section we prove our main result: that the ice crystal $\cM(\lambda)$ is isomorphic to the crystal $\cB(\lambda)$ of the irreducible $\fsl_n$-module of highest weight $\lambda$.  Our method is to show that the ice crystal has a unique highest weight element and then apply a result of Stembridge, characterizing the crystal $\cB(\lambda)$.

\begin{mydef}[Highest weight ice model]
  An element $M \in \cM(\lambda)$ such that $\te_i(M) = 0$ for all $i \in I$ is called a \emph{highest weight} element of $\cM(\lambda)$.
\end{mydef}

\begin{mylem} \label{lem_staircase}
  Suppose $M \in \cM(\lambda)$ is a highest weight $n \times s$ ice model.  Then, for all $i \in I$, there exists a $1 \leq q_i \leq s$ such that $(i,q) \in \beta(M)$ if and only if $q > q_i$. Moreover, we have $q_{i-1} \leq q_{i}$ for all $i>1$.
\end{mylem}

\begin{proof}
  We prove the lemma by induction on $i \in I$.   Considering $i=1$, since $\te_1(M)=0$ we have that either there are no boxes $(2,q)\in\beta(M)$ or the $\rightmoon$'s of all boxes $(2,q)\in\beta(M)$ cancel in $\sigma^\red_1(M)$. If there are no boxes of the form $(2,q)$, then define $q_i := s$ for all $i$.  Now assume there is a box in row $2$, calling the leftmost such box $(2,u)$. Since its $\rightmoon$ cancels, we have a $(1,v)\in\beta(M)$, with $v\leq u$, whose $\leftmoon$ cancels with the $\rightmoon$ of $(2,u)$ in $\sigma^\red_1(M)$. However, all vertices in row $1$ have plus signs as bottom edges, so we have $(1,q)\in\beta(M)$ for all $q\geq v$ since a vertex with the bottom edge as a plus and left edge as a minus is necessarily a box.

  Hence we define $q_1 = v - 1$.
  \begin{figure}[!h]
    \begin{center}
      \begin{tikzpicture}[anchorbase]
        \node at (-4,-1) {$1$};
        \node at (-4,0) {$2$};
        \node[left] at (-3.1,-1) {$\cdots$};
        \draw [dashed, thin] (-3,-1) -- (-1,-1);
        \draw [dashed, thin] (0,-1) -- (1,-1);
        \draw [dashed, thin] (2,-1) -- (3,-1);
        \draw [dashed, thin] (0,0) -- (1,0);
        \draw [dashed, thin] (2,0) -- (3,0);
        \draw [dashed, thin] (-2.5,-1.5) -- (-2.5,-0.5);
        \draw [dashed, thin] (-1.5,-1.5) -- (-1.5,-0.5);
        \draw [dashed, thin] (0.5,-1.5) -- (0.5,0.5);
        \draw [dashed, thin] (2.5,-1.5) -- (2.5,0.5);
        \draw [fill=white,ultra thin] (-3,-1) circle [radius=0.2];
        \node at (-3,-1) {$ - $};
        \draw [fill=white,ultra thin] (-2,-1) circle [radius=0.2];
        \node at (-2,-1) {$ - $};
        \draw [fill=white,ultra thin] (-1,-1) circle [radius=0.2];
        \node at (-1,-1) {$ \pm $};
        \draw [fill=white,ultra thin] (0,-1) circle [radius=0.2];
        \node at (0,-1) {$ \pm $};
        \node[left] at (-0.1,-1) {$ \cdots $};
        \draw [fill=white,ultra thin] (0,0) circle [radius=0.2];
        \node at (0,0) {$ - $};
        \node[left] at (-0.1,0) {$ \cdots $};
        \draw [fill=white,ultra thin] (1,-1) circle [radius=0.2];
        \node at (1,-1) {$ \pm $};
        \draw [fill=white,ultra thin] (1,0) circle [radius=0.2];
        \node at (1,0) {$ - $};
        \draw [fill=white,ultra thin] (2,-1) circle [radius=0.2];
        \node at (2,-1) {$ \pm $};
        \node[left] at (1.9,-1) {$ \cdots $};
        \draw [fill=white,ultra thin] (2,0) circle [radius=0.2];
        \node at (2,0) {$ \pm $};
        \node[left] at (1.9,0) {$ \cdots $};
        \draw [fill=white,ultra thin] (3,-1) circle [radius=0.2];
        \node at (3,-1) {$ - $};
        \draw [fill=white,ultra thin] (3,0) circle [radius=0.2];
        \node at (3,0) {$ - $};
        \draw [fill=white,ultra thin] (-2.5,-1.5) circle [radius=0.2];
        \node at (-2.5,-1.5) {$ + $};
        \draw [fill=white,ultra thin] (-2.5,-0.5) circle [radius=0.2];
        \node at (-2.5,-0.5) {$ + $};
        \draw [fill=white,ultra thin] (-1.5,-1.5) circle [radius=0.2];
        \node at (-1.5,-1.5) {$ + $};
        \draw [fill=white,ultra thin] (-1.5,-0.5) circle [radius=0.2];
        \node at (-1.5,-0.5) {$ \pm $};
        \draw [fill=white,ultra thin] (0.5,-1.5) circle [radius=0.2];
        \node at (0.5,-1.5) {$ + $};
        \draw [fill=white,ultra thin] (0.5,-0.5) circle [radius=0.2];
        \node at (0.5,-0.5) {$ + $};
        \draw [fill=white,ultra thin] (0.5,0.5) circle [radius=0.2];
        \node at (0.5,0.5) {$ + $};
        \draw [fill=white,ultra thin] (2.5,-1.5) circle [radius=0.2];
        \node at (2.5,-1.5) {$ + $};
        \draw [fill=white,ultra thin] (2.5,-0.5) circle [radius=0.2];
        \node at (2.5,-0.5) {$ \pm $};
        \draw [fill=white,ultra thin] (2.5,0.5) circle [radius=0.2];
        \node at (2.5,0.5) {$ \pm $};
        \node at (-2.5,-1) [rectangle,draw,fill=white] {$+$};
        \node at (0.5,0) [rectangle,draw,fill=white] {$+$};
        \draw [fill=gray] (-2.5,-1) circle[radius=0.1];
        \draw [fill=gray] (-1.5,-1) circle[radius=0.1];
        \draw [fill=gray] (0.5,-1) circle[radius=0.1];
        \draw [fill=gray] (0.5,0) circle[radius=0.1];
        \draw [fill=gray] (2.5,-1) circle[radius=0.1];
        \draw [fill=gray] (2.5,0) circle[radius=0.1];
      \end{tikzpicture}
      $\Longrightarrow$
      \begin{tikzpicture}[anchorbase]
        \node at (-4,-1) {$1$};
        \node at (-4,0) {$2$};
        \node[left] at (-3.1,-1) {$\cdots$};
        \draw [dashed, thin] (-3,-1) -- (-1,-1);
        \draw [dashed, thin] (0,-1) -- (1,-1);
        \draw [dashed, thin] (2,-1) -- (3,-1);
        \draw [dashed, thin] (0,0) -- (1,0);
        \draw [dashed, thin] (2,0) -- (3,0);
        \draw [dashed, thin] (-2.5,-1.5) -- (-2.5,-0.5);
        \draw [dashed, thin] (-1.5,-1.5) -- (-1.5,-0.5);
        \draw [dashed, thin] (0.5,-1.5) -- (0.5,0.5);
        \draw [dashed, thin] (2.5,-1.5) -- (2.5,0.5);
        \draw [fill=white,ultra thin] (-3,-1) circle [radius=0.2];
        \node at (-3,-1) {$ - $};
        \draw [fill=white,ultra thin] (-2,-1) circle [radius=0.2];
        \node at (-2,-1) {$ - $};
        \draw [fill=white,ultra thin] (-1,-1) circle [radius=0.2];
        \node at (-1,-1) {$ - $};
        \draw [fill=white,ultra thin] (0,-1) circle [radius=0.2];
        \node at (0,-1) {$ - $};
        \node[left] at (-0.1,-1) {$ \cdots $};
        \draw [fill=white,ultra thin] (0,0) circle [radius=0.2];
        \node at (0,0) {$ - $};
        \node[left] at (-0.1,0) {$ \cdots $};
        \draw [fill=white,ultra thin] (1,-1) circle [radius=0.2];
        \node at (1,-1) {$ - $};
        \draw [fill=white,ultra thin] (1,0) circle [radius=0.2];
        \node at (1,0) {$ - $};
        \draw [fill=white,ultra thin] (2,-1) circle [radius=0.2];
        \node at (2,-1) {$ - $};
        \node[left] at (1.9,-1) {$ \cdots $};
        \draw [fill=white,ultra thin] (2,0) circle [radius=0.2];
        \node at (2,0) {$ \pm $};
        \node[left] at (1.9,0) {$ \cdots $};
        \draw [fill=white,ultra thin] (3,-1) circle [radius=0.2];
        \node at (3,-1) {$ - $};
        \draw [fill=white,ultra thin] (3,0) circle [radius=0.2];
        \node at (3,0) {$ - $};
        \draw [fill=white,ultra thin] (-2.5,-1.5) circle [radius=0.2];
        \node at (-2.5,-1.5) {$ + $};
        \draw [fill=white,ultra thin] (-2.5,-0.5) circle [radius=0.2];
        \node at (-2.5,-0.5) {$ + $};
        \draw [fill=white,ultra thin] (-1.5,-1.5) circle [radius=0.2];
        \node at (-1.5,-1.5) {$ + $};
        \draw [fill=white,ultra thin] (-1.5,-0.5) circle [radius=0.2];
        \node at (-1.5,-0.5) {$ + $};
        \draw [fill=white,ultra thin] (0.5,-1.5) circle [radius=0.2];
        \node at (0.5,-1.5) {$ + $};
        \draw [fill=white,ultra thin] (0.5,-0.5) circle [radius=0.2];
        \node at (0.5,-0.5) {$ + $};
        \draw [fill=white,ultra thin] (0.5,0.5) circle [radius=0.2];
        \node at (0.5,0.5) {$ + $};
        \draw [fill=white,ultra thin] (2.5,-1.5) circle [radius=0.2];
        \node at (2.5,-1.5) {$ + $};
        \draw [fill=white,ultra thin] (2.5,-0.5) circle [radius=0.2];
        \node at (2.5,-0.5) {$ + $};
        \draw [fill=white,ultra thin] (2.5,0.5) circle [radius=0.2];
        \node at (2.5,0.5) {$ \pm $};
        \node at (-2.5,-1) [rectangle,draw,fill=white] {$+$};
        \node at (-1.5,-1) [rectangle,draw,fill=white] {$+$};
        \node at (0.5,-1) [rectangle,draw,fill=white] {$+$};
        \node at (2.5,-1) [rectangle,draw,fill=white] {$+$};
        \node at (0.5,0) [rectangle,draw,fill=white] {$+$};
        \draw [fill=gray] (-2.5,-1) circle[radius=0.1];
        \draw [fill=gray] (-1.5,-1) circle[radius=0.1];
        \draw [fill=gray] (0.5,-1) circle[radius=0.1];
        \draw [fill=gray] (0.5,0) circle[radius=0.1];
        \draw [fill=gray] (2.5,-1) circle[radius=0.1];
        \draw [fill=gray] (2.5,0) circle[radius=0.1];
      \end{tikzpicture}
      \caption{The base `stair' of $M$.}
    \end{center}
  \end{figure}
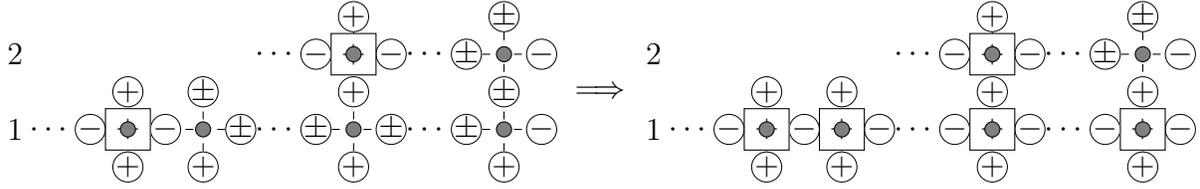

  Now, letting $1 < m \leq n$, we assume $q_1,q_2,\dotsc, q_{m-1}$ exist. If $q_{m-1} = s$ we are done, so we suppose $q_{m-1} < s$.  We assume there is a leftmost $(m,t)\in\beta(M)$, otherwise we are done. But since $\te_{m-1} M = 0$, we know its $\rightmoon$ cancels in $\sigma^\red_{m-1}(M)$, and by assumption we have that $t > q_{m-1}$. This means all vertices $M_{m,q}$ with $q\geq t$ satisfy $M_{m,q}^\downarrow = +$. Hence the same principle as in the base case applies, and $(m,q)\in\beta(M)$ for all $q > t-1 =: q_m$.
\end{proof}

Under the assumptions of Lemma~\ref{lem_staircase}, the boxes of $M$ form a ``staircase'' formation flush to the right side of the model. For each row $i$, the collection of boxes in that row is called the \textit{$i^\text{th}$ stair}, the length of the stair---the number of boxes---being $n + \lambda_1 - q_i$.  An example of this staircase formation is shown in Figure \ref{fig_staircase}.

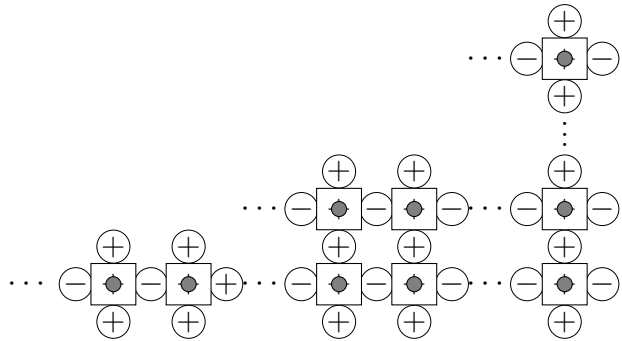
\begin{figure}[!h]
  \begin{center}
    \begin{tikzpicture}
      \draw [dashed, thin] (-3,-1) -- (-1,-1);
      \draw [dashed, thin] (0,-1) -- (2,-1);
      \draw [dashed, thin] (0,0) -- (2,0);
      \draw [dashed, thin] (3,-1) -- (4,-1);
      \draw [dashed, thin] (3,0) -- (4,0);
      \draw [dashed, thin] (3,2) -- (4,2);
      \draw [dashed, thin] (-2.5,-1.5) -- (-2.5,-0.5);
      \draw [dashed, thin] (-1.5,-1.5) -- (-1.5,-0.5);
      \draw [dashed, thin] (0.5,-1.5) -- (0.5,0.5);
      \draw [dashed, thin] (1.5,-1.5) -- (1.5,0.5);
      \draw [dashed, thin] (3.5,-1.5) -- (3.5,0.5);
      \draw [dashed, thin] (3.5,1.5) -- (3.5,2.5);
      \draw [fill=white,ultra thin] (-3,-1) circle [radius=0.22];
      \node at (-3,-1) {$ - $};
      \node[left] at (-3.2,-1) {$ \cdots $};
      \draw [fill=white,ultra thin] (-2,-1) circle [radius=0.22];
      \node at (-2,-1) {$ - $};
      \draw [fill=white,ultra thin] (-1,-1) circle [radius=0.22];
      \node at (-1,-1) {$ + $};
      \draw [fill=white,ultra thin] (0,-1) circle [radius=0.22];
      \node at (0,-1) {$ - $};
      \node[left] at (-0.1,-1) {$ \cdots $};
      \draw [fill=white,ultra thin] (0,0) circle [radius=0.22];
      \node at (0,0) {$ - $};
      \node[left] at (-0.1,0) {$ \cdots $};
      \draw [fill=white,ultra thin] (1,-1) circle [radius=0.22];
      \node at (1,-1) {$ - $};
      \draw [fill=white,ultra thin] (1,0) circle [radius=0.22];
      \node at (1,0) {$ - $};
      \draw [fill=white,ultra thin] (2,-1) circle [radius=0.22];
      \node at (2,-1) {$ - $};
      \draw [fill=white,ultra thin] (2,0) circle [radius=0.22];
      \node at (2,0) {$ - $};
      \draw [fill=white,ultra thin] (3,-1) circle [radius=0.22];
      \node at (3,-1) {$ - $};
      \node[left] at (2.9,-1) {$ \cdots $};
      \draw [fill=white,ultra thin] (3,0) circle [radius=0.22];
      \node at (3,0) {$ - $};
      \node[left] at (2.9,0) {$ \cdots $};
      \draw [fill=white,ultra thin] (3,2) circle [radius=0.22];
      \node at (3,2) {$ - $};
      \node[left] at (2.9,2) {$ \cdots $};
      \draw [fill=white,ultra thin] (4,-1) circle [radius=0.22];
      \node at (4,-1) {$ - $};
      \draw [fill=white,ultra thin] (4,0) circle [radius=0.22];
      \node at (4,0) {$ - $};
      \draw [fill=white,ultra thin] (4,2) circle [radius=0.22];
      \node at (4,2) {$ - $};
      \draw [fill=white,ultra thin] (-2.5,-1.5) circle [radius=0.22];
      \node at (-2.5,-1.5) {$ + $};
      \draw [fill=white,ultra thin] (-2.5,-0.5) circle [radius=0.22];
      \node at (-2.5,-0.5) {$ + $};
      \draw [fill=white,ultra thin] (-1.5,-1.5) circle [radius=0.22];
      \node at (-1.5,-1.5) {$ + $};
      \draw [fill=white,ultra thin] (-1.5,-0.5) circle [radius=0.22];
      \node at (-1.5,-0.5) {$ + $};
      \draw [fill=white,ultra thin] (0.5,-1.5) circle [radius=0.22];
      \node at (0.5,-1.5) {$ + $};
      \draw [fill=white,ultra thin] (0.5,-0.5) circle [radius=0.22];
      \node at (0.5,-0.5) {$ + $};
      \draw [fill=white,ultra thin] (0.5,0.5) circle [radius=0.22];
      \node at (0.5,0.5) {$ + $};
      \draw [fill=white,ultra thin] (1.5,-1.5) circle [radius=0.22];
      \node at (1.5,-1.5) {$ + $};
      \draw [fill=white,ultra thin] (1.5,-0.5) circle [radius=0.22];
      \node at (1.5,-0.5) {$ + $};
      \draw [fill=white,ultra thin] (1.5,0.5) circle [radius=0.22];
      \node at (1.5,0.5) {$ + $};
      \draw [fill=white,ultra thin] (3.5,-1.5) circle [radius=0.22];
      \node at (3.5,-1.5) {$ + $};
      \draw [fill=white,ultra thin] (3.5,-0.5) circle [radius=0.22];
      \node at (3.5,-0.5) {$ + $};
      \draw [fill=white,ultra thin] (3.5,0.5) circle [radius=0.22];
      \node at (3.5,0.5) {$ + $};
      \draw [fill=white,ultra thin] (3.5,1.5) circle [radius=0.22];
      \node at (3.5,1.5) {$ + $};
      \draw [fill=white,ultra thin] (3.5,2.5) circle [radius=0.22];
      \node at (3.5,2.5) {$ + $};
      \node at (-2.5,-1) [rectangle,draw,fill=white] {$+$};
      \node at (-1.5,-1) [rectangle,draw,fill=white] {$+$};
      \node at (0.5,-1) [rectangle,draw,fill=white] {$+$};
      \node at (0.5,0) [rectangle,draw,fill=white] {$+$};
      \node at (1.5,-1) [rectangle,draw,fill=white] {$+$};
      \node at (1.5,0) [rectangle,draw,fill=white] {$+$};
      \node at (3.5,0) [rectangle,draw,fill=white] {$+$};
      \node at (3.5,-1) [rectangle,draw,fill=white] {$+$};
      \node at (3.5,2) [rectangle,draw,fill=white] {$+$};
      \node at (3.5,1.1) {$\vdots$};
      \draw [fill=gray] (-2.5,-1) circle[radius=0.1];
      \draw [fill=gray] (-1.5,-1) circle[radius=0.1];
      \draw [fill=gray] (0.5,-1) circle[radius=0.1];
      \draw [fill=gray] (0.5,0) circle[radius=0.1];
      \draw [fill=gray] (1.5,-1) circle[radius=0.1];
      \draw [fill=gray] (1.5,0) circle[radius=0.1];
      \draw [fill=gray] (3.5,-1) circle[radius=0.1];
      \draw [fill=gray] (3.5,0) circle[radius=0.1];
      \draw [fill=gray] (3.5,2) circle[radius=0.1];
		\end{tikzpicture}
    \caption{The staircase of boxes.}\label{fig_staircase}
  \end{center}
\end{figure}

\begin{mydef}[Diagonal of minuses]  \label{def_diagonalofminuses}
  Suppose $M\in\cM(\lambda)$ is an $n\times s$ ice model.  We say that there is a \emph{diagonal of minuses} above $(p,q)\in\beta(M)$ if and only if $M_{p+k,q-1- k}^\uparrow = -$ for all $0 \leq k \leq \min\{q-1,n-p\}$. See Figure \ref{diagminus}.

  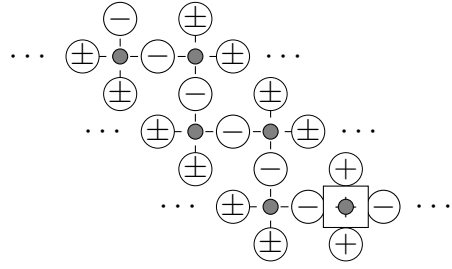
\begin{figure}[!h]
    \begin{center}
      \begin{tikzpicture}
        \draw [dashed, thin] (0,-1) -- (2,-1);
        \draw [dashed, thin] (-1,0) -- (1,0);
        \draw [dashed, thin] (-2,1) -- (0,1);
        \draw [dashed, thin] (-1.5,0.5) -- (-1.5,1.5);
        \draw [dashed, thin] (-0.5,-0.5) -- (-0.5,1.5);
        \draw [dashed, thin] (0.5,-1.5) -- (0.5,0.5);
        \draw [dashed, thin] (1.5,-1.5) -- (1.5,-0.5);
        \draw [fill=white,ultra thin] (-2,1) circle [radius=0.22];
        \node at (-2,1) {$ \pm $};
        \node at (-2.7,1) {$ \cdots $};
        \draw [fill=white,ultra thin] (-1,0) circle [radius=0.22];
        \node at (-1,0) {$ \pm $};
        \draw [fill=white,thin] (-1,1) circle [radius=0.22];
        \node at (-1,1) {$ - $};
        \node at (-1.7,0) {$ \cdots $};
        \draw [fill=white,ultra thin] (0,-1) circle [radius=0.22];
        \node at (0,-1) {$ \pm $};
        \draw [fill=white,thin] (0,0) circle [radius=0.22];
        \node at (0,0) {$ - $};
        \draw [fill=white,ultra thin] (0,1) circle [radius=0.22];
        \node at (0,1) {$ \pm $};
        \node at (-0.7,-1) {$ \cdots $};
        \node at (0.7,1) {$ \cdots $};
        \draw [fill=white,thin] (1,-1) circle [radius=0.22];
        \node at (1,-1) {$ - $};
        \draw [fill=white,ultra thin] (1,0) circle [radius=0.22];
        \node at (1,0) {$ \pm $};
        \node at (1.7,0) {$ \cdots $};
        \draw [fill=white,thin] (2,-1) circle [radius=0.22];
        \node at (2,-1) {$ - $};
        \node at (2.7,-1) {$ \cdots $};
        \draw [fill=white,ultra thin] (-1.5,0.5) circle [radius=0.22];
        \node at (-1.5,0.5) {$ \pm $};
        \draw [fill=white,thin] (-1.5,1.5) circle [radius=0.22];
        \node at (-1.5,1.5) {$ - $};
        \draw [fill=white,ultra thin] (-0.5,-0.5) circle [radius=0.22];
        \node at (-0.5,-0.5) {$ \pm $};
        \draw [fill=white,thin] (-0.5,0.5) circle [radius=0.22];
        \node at (-0.5,0.5) {$ - $};
        \draw [fill=white,ultra thin] (-0.5,1.5) circle [radius=0.22];
        \node at (-0.5,1.5) {$ \pm $};
        \draw [fill=white,ultra thin] (0.5,-1.5) circle [radius=0.22];
        \node at (0.5,-1.5) {$ \pm $};
        \draw [fill=white,thin] (0.5,-0.5) circle [radius=0.22];
        \node at (0.5,-0.5) {$ - $};
        \draw [fill=white,ultra thin] (0.5,0.5) circle [radius=0.22];
        \node at (0.5,0.5) {$ \pm $};
        \draw [fill=white,ultra thin] (1.5,-1.5) circle [radius=0.22];
        \node at (1.5,-1.5) {$ + $};
        \draw [fill=white,ultra thin] (1.5,-0.5) circle [radius=0.22];
        \node at (1.5,-0.5) {$ + $};
        \node at (1.5,-1) [rectangle,draw,fill=white] {$+$};
        \draw [fill=gray] (-1.5,1) circle[radius=0.1];
        \draw [fill=gray] (-0.5,0) circle[radius=0.1];
        \draw [fill=gray] (-0.5,1) circle[radius=0.1];
        \draw [fill=gray] (0.5,-1) circle[radius=0.1];
        \draw [fill=gray] (0.5,0) circle[radius=0.1];
        \draw [fill=gray] (1.5,-1) circle[radius=0.1];
			\end{tikzpicture}
      \caption{Example of a diagonal of minuses.\label{diagminus}}
    \end{center}
  \end{figure}
\end{mydef}

\begin{mylem} \label{lem_minusdiagonal}
  Suppose $M \in \cM(\lambda)$ is a highest weight $n \times s$ ice model and $i \in I$.  If $(i,q)\in\beta(M)$ is the leftmost box in the $i^\text{th}$ stair, then there is a diagonal of minuses above $(i,q)$.
\end{mylem}

\begin{proof}
  If the $p^\text{th}$ stair is of length $0$, then $M_{p,s}^\uparrow = -$ and we see that locally $M$ appears as in Figure \ref{fig_ps}.
  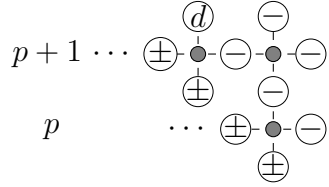
\begin{figure}[!h]
  	\begin{center}
      \begin{tikzpicture}
        \draw [dashed, thin] (-2,-1) -- (-1,-1);
        \draw [dashed, thin] (-3,0) -- (-1,0);
        \draw [dashed, thin] (-2.5,-0.5) -- (-2.5,0.5);
        \draw [dashed, thin] (-1.5,-1.5) -- (-1.5,0.5);
        \draw [fill=white,ultra thin] (-3,0) circle [radius=0.22];
        \node at (-3,0) {$ \pm $};
        \draw [fill=white,ultra thin] (-2,-1) circle [radius=0.2];
        \node at (-2,-1) {$ \pm $};
        \draw [fill=white,ultra thin] (-2,0) circle [radius=0.2];
        \node at (-2,0) {$ - $};
        \draw [fill=white,ultra thin] (-1,-1) circle [radius=0.2];
        \node at (-1,-1) {$ - $};
        \draw [fill=white,ultra thin] (-1,0) circle [radius=0.2];
        \node at (-1,0) {$ - $};
        \draw [fill=white,ultra thin] (-1.5,-1.5) circle [radius=0.2];
        \node at (-1.5,-1.5) {$ \pm $};
        \draw [fill=white,ultra thin] (-1.5,-0.5) circle [radius=0.2];
        \node at (-1.5,-0.5) {$ - $};
        \draw [fill=white,ultra thin] (-1.5,0.5) circle [radius=0.2];
        \node at (-1.5,0.5) {$ - $};
        \draw [fill=white,ultra thin] (-2.5,-0.5) circle [radius=0.2];
        \node at (-2.5,-0.5) {$ \pm $};
        \draw [fill=white,ultra thin] (-2.5,0.5) circle [radius=0.2];
        \node at (-2.5,0.5) {$ d $};
        \draw [fill=gray] (-2.5,0) circle[radius=0.1];
        \draw [fill=gray] (-1.5,-1) circle[radius=0.1];
        \draw [fill=gray] (-1.5,0) circle[radius=0.1];	
        \node[left] at (-3.2,0) {$ \cdots $};
        \node[left] at (-3.9,0) {$ p+1 $};
        \node[left] at (-2.2,-1) {$ \cdots $};
        \node[left] at (-4.2,-1) {$ p $};
      \end{tikzpicture}
      \caption{Local configuration around $M_{p,s}$.\label{fig_ps}}
    \end{center}
  \end{figure}

  If $d=+$, then we have that $(p+1,s-1)\in\beta(M)$, which would be a contradiction as the $(p+1)^\text{st}$ stair would be longer than the $p^\text{th}$. Hence $d=-$. Applying the same reasoning yields that $M_{p+k,s-k}^\uparrow = -$ for all $0\leq k < \min\{s,n-p\}$.

  Otherwise, suppose $(p,q)\in\beta(M)$ is the leftmost box in a stair of $M$. Since $(p,q-1)\not\in\beta(M)$, $M_{p,q-1}^\uparrow = -$, and we may apply the same reasoning as above.
\end{proof}

\begin{mylem} \label{lem_stairbound}
  Suppose $M\in\cM(\lambda)$ is a highest weight $n\times s$ ice model. If $(p,q) \in\beta(M)$, then $q \geq n - p + 2$.
\end{mylem}

\begin{proof}
  It suffices to show this condition for the leftmost box. Suppose, towards a contradiction, that $q < n - p + 2$ for the leftmost box $(p,q)$ of the $p^\text{th}$ stair.  This implies that $q-1 \leq n-p$---that is, that the horizontal distance of the box from the left is at most the vertical distance to the top of the ice model. If this were the case, the diagonal of minuses above $(p,q)$ would reach a vertex $M_{t,1}$ for some $p \leq t < n$. But if $M_{t,1}^\uparrow = -$, then $M_{t+1,1}^\leftarrow = -$, contradicting the boundary conditions.
\end{proof}

\begin{mycor} \label{cor_diagonaltothetop}
  Suppose $M \in \cM(\lambda)$ is a highest weight $n\times s$ ice model and $i \in I$.  If $(i,q) \in \beta(M)$ is the leftmost box in the $i^\text{th}$ stair, then $M_{i+k,q-1-k}^\uparrow = -$ for all $0\leq k\leq n-i$. That is, the diagonal of minuses above $(i,q)$ reaches the top of the ice model.
\end{mycor}

The next lemma shows that for a highest weight ice model $M$, the minus signs in the top boundary condition are in bijective correspondence with the stairs.

\begin{mylem} \label{lem_stairbijection}
  If $M\in\cM(\lambda)$ is a highest weight $n \times s$ ice model, then the set of stairs in $M$ (including those of length $0$) is in bijective correspondence with the set
  \[
    \{1 \leq q \leq s \mid M_{n,q}^\uparrow = - \}
  \]
  of minuses at the top of $M$.  Furthermore, $\wt(M) = \lambda$.
\end{mylem}

\begin{proof}
  By Corollary \ref{cor_diagonaltothetop}, we have that every stair corresponds to a minus sign on the top boundary. We will now show the converse.

  Suppose we have $M_{n,q}^\uparrow = -$ and that $M_{n,q}$ is not a part of a diagonal of minuses above some box $(p,q)\in\beta(M)$.   Then locally we have a configuration as in Figure \ref{fig_minustop}.
  \begin{figure}[!h]
  	\begin{center}
      \begin{tikzpicture}
        \draw [dashed, thin] (-3,-1) -- (-1,-1);
        \draw [dashed, thin] (-2.5,-1.5) -- (-2.5,-0.5);
        \draw [dashed, thin] (-1.5,-1.5) -- (-1.5,-0.5);
        \draw [fill=white,ultra thin] (-3,-1) circle [radius=0.2];
        \node at (-3,-1) {$ \pm $};
        \draw [fill=white,ultra thin] (-2,-1) circle [radius=0.2];
        \node at (-2,-1) {$ - $};
        \draw [fill=white,ultra thin] (-1,-1) circle [radius=0.2];
        \node at (-1,-1) {$ \pm $};
        \draw [fill=white,ultra thin] (-1.5,-1.5) circle [radius=0.2];
        \node at (-1.5,-1.5) {$ c $};
        \draw [fill=white,ultra thin] (-1.5,-0.5) circle [radius=0.2];
        \node at (-1.5,-0.5) {$ b $};
        \draw [fill=white,ultra thin] (-2.5,-1.5) circle [radius=0.2];
        \node at (-2.5,-1.5) {$ \pm $};
        \draw [fill=white,ultra thin] (-2.5,-0.5) circle [radius=0.2];
        \node at (-2.5,-0.5) {$ - $};
        \draw [fill=gray] (-2.5,-1) circle[radius=0.1];
        \draw [fill=gray] (-1.5,-1) circle[radius=0.1];
        \node[left] at (-3.2,-1) {$ \cdots $};
        \node[right] at (-0.8,-1) {$ \cdots $};
        \node[below] at (-2,-1.5) {$ \vdots $};
      \end{tikzpicture}
      \caption{Local configuration around a minus, including the vertices $M_{n,q}$ and $M_{n,q+1}$.\label{fig_minustop}}
    \end{center}
  \end{figure}

  If $c=+$, then $(n,q+1)\in\beta(M)$ and we get a contradiction since then $M_{n,q}$ belongs to a diagonal of minuses above the box $(n,q+1)$.  Therefore, we suppose $c=-$.  Continuing in this manner, we will eventually reach a vertex in row $1$ whose bottom edge is necessarily a plus, and hence the vertex is a box.  This completes the proof of the first assertion of the lemma.

  We now prove that $\wt(M) = \lambda$.  It follows from the above and Corollary~\ref{cor_diagonaltothetop} that the leftmost box $(i,q)$ in the $i^\text{th}$ stair corresponds to the minus $M_{n,q-1-n+i}^\uparrow$ at the top of the model.  By Definition~\ref{def_boundary}\pref{def:boundary-top-minuses}, this implies that $q-1-n+i = \lambda_1 - \lambda_i + i$.  Hence, the number of boxes in the $i^\text{th}$ stair is
  \[
    \lambda_1 + n - q + 1 = \lambda_i.
  \]
  It then follows immediately from Definition~\ref{def_wt} that $\wt(M)=\lambda$.
\end{proof}

\begin{myprop} \label{prop_uniquehighestweight}
  For every partition $\lambda = (\lambda_1,\lambda_2,\dotsc,\lambda_n)$, there exists a unique highest weight ice model $M \in \cM(\lambda)$.
\end{myprop}

\begin{proof}
  It follows from \pref{C2} and the fact that the set $\cM(\lambda)$ is finite that there exists a highest weight ice model $M \in \cM(\lambda)$.  It remains to show that the highest weight ice model is unique.

  Referring to the valid vertex configurations in Figure~\ref{validvertices}, we see that, excluding the vertex configuration that corresponds to a box (type 2), the four remaining configurations are all uniquely determined by their right and bottom edges.  Since, by Lemma~\ref{lem_stairbijection}, we know the exact locations of every box in our model, we may identify every other vertex uniquely using these two edges, inductively starting from the bottom-right vertex and identifying them right-to-left, bottom-to-top. Thus, $M$ is uniquely determined by $\lambda$.
\end{proof}

\begin{theorem} \label{blambda}
  For every partition $\lambda = (\lambda_1,\lambda_2,\dotsc,\lambda_n)$, the ice crystal $\cM(\lambda)$ is isomorphic to the irreducible highest weight crystal $\cB(\lambda)$ corresponding to the irreducible highest weight representation of $\fsl_n$ with highest weight $\lambda$.
\end{theorem}

\begin{proof}
  The ice crystal $\cM(\lambda)$ is regular by Proposition~\ref{prop_regular} and has a unique highest weight element of weight $\lambda$ by Proposition~\ref{prop_uniquehighestweight}.  The result then follows immediately from \cite[Th.~3.2, Th.~3.3]{Stembridge}.
\end{proof}

%
%

\bibliographystyle{alphaurl}
\bibliography{ice-crystals}

\end{document}